\documentclass{article}
\usepackage[german, english]{babel}
\usepackage{amssymb,amsmath}
\usepackage{amsmath, amsthm}
\usepackage{amsfonts,epsfig}
\newtheoremstyle{theorem}
  {10pt}		  
  {10pt}  
  {\sl}  
  {\parindent}     
  {\bf}  
  {. }    
  { }    
  {}     
\theoremstyle{theorem}
\newtheorem{theorem}{Theorem}
\newtheorem{corollary}[theorem]{Corollary}
\newtheorem{proposition}[theorem]{Proposition}
\newtheorem{remark}[theorem]{Remark}
\newtheorem{example}[theorem]{Example}
\newtheorem{lemma}[theorem]{Lemma}

\newtheoremstyle{defi}
  {10pt}		  
  {10pt}  
  {\rm}  
  {\parindent}     
  {\bf}  
  {. }    
  { }    
  {}     
\theoremstyle{defi}
\newtheorem{definition}[theorem]{Definition}

\newcommand{\naturals}{\mathbb{N}}

\newcommand{\Right}{\mathfrak{R}}
\newcommand{\Left}{\mathfrak{L}}
\newcommand{\maxs}{\ell}
\newcommand{\nul}{{\rm nul}}

\newcommand{\defic}{{\rm def}}
\newcommand{\RS}{{\rm RS}}
\newcommand{\NS}{{\rm NS}}
\newcommand{\Ker}{{\rm Ker}}
\newcommand{\Img}{{\rm Im}}
\newcommand{\Class}{\mathbb{L}}
\newcommand{\card}{{\rm card}}
\newcommand{\Span}{{\rm span}}
\newcommand{\RFM}{\rm \mathbb{R}\mathbb{F}\mathbb{M}_\omega({\mathcal{F}})}
\newcommand{\Hom}{{\rm Hom}_{\mathcal{F}}({\mathcal{F}}^{(\omega)})}
\begin{document}

\bibliographystyle{plain}

\title{The Infinite Gauss-Jordan Elimination on \\ Row-Finite $\omega\times\omega$ Matrices}

\author{A. G. Paraskevopoulos\\
\small The Center for Research and Applications of Nonlinear Systems (CRANS)\\
\small Department of Mathematics, Division of Applied Analysis,\\
\small University of Patras, 26500 Patra, Greece
}

\maketitle

\begin{abstract} The Gauss-Jordan elimination algorithm is extended to reduce a row-finite $\omega\times\omega$ matrix to lower row-reduced form, founded on a strategy of rightmost pivot elements. Such reduced matrix form preserves row equivalence, unlike the dominant (upper) row-reduced form. This algorithm provides a constructive alternative to an earlier existence and uniqueness result for Quasi-Hermite forms based on the axiom of countable choice. As a consequence, the general solution of an infinite system of linear equations with a row-finite coefficient $\omega\times\omega$ matrix is fully constructible.\\


\end{abstract}

\section{Introduction} \label{intro-sec} 
The Gaussian elimination, the fundamental tool in reducing matrix problems to simpler ones, from the ancient times (see~\cite{Ha:Ch}) up to the modern age (see~\cite{Gr:GE}) is dominated by strategies founded on leftmost pivot elements. When it is supplemented with Jordan elimination, it generates the upper row-reduced echelon form (URREF) of a matrix. This is the most widespread row canonical form of finite matrices, within the meaning that every finite matrix is reducible to a unique URREF, no matter what method is used. Even though the lower row-reduced echelon form (LRREF) is also row canonical, the URREF has appropriated the entire use of the term``row-reduced echelon form''. As it turns out, however, leftmost pivot strategies (LPS) along with matrices in URREF are serious obstacles in extending the Gauss-Jordan elimination to cover the algebra $\RFM$ of  row-finite $\omega\times\omega$ matrices over a field $\mathcal{F}$.

In all that follows, the shorthand notation LRRF (resp. URRF) stands for a lower (resp. upper) row-reduced form, without requiring the echelon conditions. Following earlier work of Toeplitz~\cite{To:Auf}, Fulkerson introduced in~\cite{Fu:Th}\footnote{I thank Travis D. Warwick, Librarian of the Stephen Cole Kleene Mathematics Library at University of Wisconsin - Madison, for making Fulkerson's Thesis available to me.} the notion ``Quasi-Hermite form" (QHF)  of a row-finite $\omega\times\omega$ matrix, establishing there existence and ``almost" uniqueness results largely ignored for a long period. These results, briefly presented in Section \ref{sec:PreliminaryNotationAndResults}, are one of the main sources of motivation of this study. As shown in Section \ref{sec:RowFiniteMatricesInLRRF}, a QHF is a row-finite $\omega\times\omega$ matrix in LRRF, the nonzero rows of which are of strictly increasing row-length. Let $C\in \RFM$. By the ``existence of a QHF of $C$" it is meant that there is a QHF, say $H\in \RFM$, such that $C$ and $H$ are left associates (or row equivalent). Fulkerson's proof on the existence of a QHF of $C$, as implicitly stated in~\cite[Theorem 3.1]{Fu:Th}, necessarily involves the axiom of countable choice (see Section \ref{sec:PreliminaryNotationAndResults}). It enabled him to overcome the lack of a certain rule for choosing a complete basis of row-length, say $\mbox{\boldmath$\mathrm{A}$}=(\mbox{\boldmath$\mathrm{A}$}_{k})_k$, from the length-equivalence classes of the row space of $C$. The nonzero rows of $H$ are connected by a recurrence relation with the mediation of elements of $\mbox{\boldmath$\mathrm{A}$}$. Because of the non-constructive nature of the axiom of countable choice, the presence of this axiom prior to the recurrence, reduces the constructibility of this recurrence to cases in which a complete basis of row-length is given in advance by $C$ (see Examples \ref{example2}, \ref{example3}). By way of contrast, the algorithm proposed here provides a certain rule for choosing a complete basis of row-length (see Propositions \ref{prop4}, \ref{prop6}), establishing the constructiveness of QHF of $C$.

The implementation of the Gauss-Jordan elimination with rightmost pivots is referred to as rightmost pivot strategy (RPS). The infinite Gauss-Jordan algorithm (see Section \ref{sec:TheInfiniteGauss-JordanEliminationAlgorithm}) consists of an effective routine implemented with RPS generating a chain of matrices with finite number of rows in LRRF. It starts with the reduction of the top submatrix of an arbitrary $C\in \RFM$ extended up to and including the second nonzero row of $C$, and then repeatedly apply to reduce matrices comprising the previously reduced matrix augmented by successive rows of $C$ one at a time and ad infinitum. 
This avoids the apparent difficulty due to the direct application of the Gauss Jordan algorithm to the entire $C$, which generally requires an infinite number of row operations to eliminate the nonzero entries of a pivot column. The question now turns to whether the above-mentioned infinite chain of row-reduced matrices, upon algorithm completion, ends up as a matrix in LRRF row equivalent to $C$.

The infinite Gauss-Jordan algorithm implemented with LPS has consequences in the row reduction of $\omega\times\omega$ integer valued row-finite matrices modulo a prime power (see \cite{Os:Pr}). As a general strategy, however (see Example \ref{example1}), LPS forces us to choose from two undesirable alternatives: Either that the infinite elimination leads to an undecidable problem or that it ends up as a matrix in URRF, but not preserving row equivalence. Apart from the case of invertible row-finite matrices (see the reference previously cited), it is shown here that there is no row equivalent URRF for an extensive class of row-finite matrices. However, the implementation of the infinite Gauss-Jordan algorithm with priority to RPS makes it possible to overcome the impasse regarding LPS. It is proved in Theorem \ref{main-theorem} that as the process proceeds, after a sufficiently large but finite number of algorithmic steps, every row remains invariant. As a consequence, upon algorithm completion, a row equivalent matrix in LRRF is accomplished (see Section \ref{sec:TheInfiniteGJEliminationAndTheRowEquivalence}). It is further shown the existence of QHF, without invoking the countable axiom of choice. 

Matrices in LRREF require that all zero rows are grouped on the top (resp. bottom) of the matrix. Thus, the LRREF of $C\in \RFM$, having infinite nullity (resp. rank) and nonzero rank (resp. nullity), must have at least the ordinal $\omega+1$ as row indexing set (see Example \ref{example grouping zero rows}), thus not belonging to the algebra $\RFM$. Within this algebra, however, an almost uniqueness result is feasible, as established in~\cite{Fu:Th} for QHFs. If the nullity of $C$ is finite a complete uniqueness result for the LRREF of $C$ follows directly (see Example \ref{example3}).

The full replacement of the axiom of countable choice by the infinite Gauss-Jordan algorithm implemented with RPS enables us to infer the existence of a LRRF and a QHF from the constructiveness of these forms, just as in the finite dimensional case. In Section \ref{sec:ConstructionOfQuasiHermiteForms}, the infinite Gauss-Jordan algorithm is supplemented with a reordering routine for constructing a chain of submatrices of a QHF of the original matrix. Moreover, an alternative algorithmic scheme is proposed, which establishes links with the currently developed software.

A remarkable application of the algorithmic process, proposed in this paper, concerns the construction of the general solution for infinite systems of linear equations associated with an arbitrary row-finite coefficient $\omega\times\omega$ matrix, which is discussed in the last Section of this paper.

\section{Preliminary Notation and Results}
\label{sec:PreliminaryNotationAndResults}
Throughout this paper, the set of natural numbers $\naturals=\{0,1,2,...\}$ endowed with the standard well ordering by magnitude is identified, as usual, with the ordinal number $\omega$. Also, $\mathcal{F}$ stands for an algebraic field and ${\mathcal{F}}^{(\omega)}$ for the vector space of sequences in $\mathcal{F}$ of only finitely many nonzero terms. The canonical basis of ${\mathcal{F}}^{(\omega)}$ is denoted by $\mbox{\boldmath $\mathrm{E}$}=\{\mbox{\boldmath $\mathrm{e}$}_i\}_{i\in \omega}$, where $\mbox{\boldmath $\mathrm{e}$}_i=(\delta_{ij})_{j\in\omega}$ and $\delta_{ij}$ is the Kronecker delta function. 
A row-finite $\omega\times\omega$ matrix $A$ is a sequence of vectors in ${\mathcal{F}}^{(\omega)}$ indexed by $\omega$ and denoted by $A=(A_k)_{k\in \omega}=(a_{ki})_{(k,i) \in \omega \times \omega}$. Also, $\textbf{I}=(\mbox{\boldmath $\mathrm{e}$}_i)_{i\in \omega}$ stands for the identity $\omega\times\omega$ matrix.
A row $A_k\in {\mathcal{F}}^{(\omega)}$ of $A$ is written as:
\begin{equation} \label{equation_0}
A_k=(a_{k0}, a_{k1},...,a_{k\rho_{k}},0,0...) = \sum^{\rho_k}_{i=0} a_{ki}\mbox{\boldmath $\mathrm{e}$}_i
\end{equation} 
Using set notation, we write $A=\{(k,A_k): k\in \omega, A_k\in {\mathcal{F}}^{(\omega)}\}\subset \omega\times {\mathcal{F}}^{(\omega)}$ by means of which repeated rows become distinguishable. The set $({{\mathcal{F}}^{(\omega)}})^\omega$ of row-finite $\omega\times\omega$ matrices, equipped with the matrix addition and multiplication by scalars, turns into a linear space over $\mathcal{F}$. If $({{\mathcal{F}}^{(\omega)}})^\omega$ is further equipped with matrix multiplication, linear and ring structures turn $({{\mathcal{F}}^{(\omega)}})^\omega$ into an associative and noncommutative algebra with identity element the matrix $\textbf{I}$, denoted by $\RFM$. For a general reference on nontrivial definitions and results on matrix rings over arbitrary indexing sets, the reader is referred to \cite{AF:rings} and to the literature cited therein.

Let $A\in \RFM$ and $A_{k}=(a_{ki})_{i\in \omega}$ be a row of $A$. Following Toeplitz~\cite{To:Auf}, in view of (\ref{equation_0}), the index $\rho_{k}$ is called \emph{length}\footnote{The ``length'' can also be viewed as the distance of $\rho_{k}$ from the zero-index, which justifies this terminology. However, this quantity is an element of $\omega$, thus being itself a finite ordinal. This is not to be confused with the ``length of a vector'' defined as a Euclidean distance.} (or \emph{rightmost index}) of $A_{k}$ and alternatively will be denoted by $\maxs(A_{k})$.
The \emph{rightmost} nonzero coefficient $a_{k \rho_{k}}$ of $A_{k}$ will be also termed \emph{right leading coefficient} of $A_{k}$. Certainly $\maxs(\mbox{\boldmath $\mathrm{e}$}_i)=i$ for $i\in \omega$, thus $\maxs(\mbox{\boldmath $\mathrm{e}$}_0)=0$. As $\maxs(\textbf{0})$ does not exist in $\omega$, the convention $\maxs(\textbf{0})=-1$ is adopted. Let $\mathfrak{X}\subset {\mathcal{F}}^{(\omega)}$. The set $\Right(\mathfrak{X})=\{\maxs(x):\ x\in \mathfrak{X}\}$ will be called set of \emph{rightmost index} of $\mathfrak{X}$. Let $F$ be a sequence of nonzero vectors of ${\mathcal{F}}^{(\omega)}$ indexed by $J\subset \omega$. Evidently the \emph{map of rightmost index} of $F$, that is $\rho: J\mapsto \Right(F):\rho(j)=\maxs(F_j)$, is surjective and defines the sequence of \emph{rightmost index} of $F$ denoted by $\rho=(\rho_{j})_{j\in J}$. If $\rho$ is injective, then the indexed set notation $\{\rho_{j}\}_{j\in J}$ is used along with the standard identification of $\{\rho_{j}\}_{j\in J}$ with $\Right(F)$, when the former is treated as the image set of $\rho$. By analogy we define the \emph{leftmost index} of $A_k$ denoted by $\zeta(A_k)$, $k\in J$ or simply $\zeta_k$ along with the \emph{left leading coefficient} $a_{k\zeta_{k}}$ of $A_k$.
The convention $\zeta(\textbf{0})=-1$ is also adopted. The sequence and the set of leftmost index of $A$ are denoted by $\zeta=(\zeta_{j})_{j\in J}$ and $\Left(A)$, respectively. From now onwards, the indexing sets of zero and nonzero rows of $A\in \RFM$ are denoted by $W, J$ respectively.
\paragraph{Matrix Representations of Linear Mappings on ${\mathcal{F}}^{(\omega)}$.}
\label{sec:MatrixRepresentationsOfLinearMappings} 
In all that follows the space of linear mappings from ${\mathcal{F}}^{(\omega)}$ to ${\mathcal{F}}^{(\omega)}$ over ${\mathcal{F}}$ will be denoted by $\Hom$.  Let $f\in \Hom$. The \emph{kernel} and the \emph{range} of $f$ will be denoted by $\Ker(f)$ and $\Img(f)$ respectively. By ${\mathcal{X}}=\{\chi_k\}_{k\in\omega}$, $\mbox{\boldmath $\mathrm{B}$}=\{\mbox{\boldmath $\mathrm{b}$}_i\}_{i\in\omega}$, we shall refer to Hamel bases (or bases) of the domain and codomain of $f$ respectively. The matrix representation of $f$, relative to $({\mathcal{X}}, \mbox{\boldmath $\mathrm{B}$})$, is a row-finite matrix  $A=(a_{ki})_{(k,i) \in \omega \times \omega}$ with entries determined by $f(\chi_{k})=\sum_i a_{ki} \mbox{\boldmath $\mathrm{b}$}_i$ and we shall referred to it as $A=[f]^{\mbox{\scriptsize \boldmath $\mathrm{B}$}}_{{\mathcal{X}}}$. On the other hand, if $A\in\RFM$, there is a unique linear extension of $\chi_{k}\mapsto\sum_i a_{ki} \mbox{\boldmath 
$\mathrm{b}$}_i$ to $\Hom$, call it $f$, such that $A=[f]^{\mbox{\scriptsize \boldmath $\mathrm{B}$}}_{{\mathcal{X}}}$. A row $A_k$ of $A$ represents the element $f(\chi_k)$, relative to $\mbox{\boldmath $\mathrm{B}$}$\vspace{0.02in}. 
 
The matrix multiplication on $\RFM$ is defined in complete analogy with finite matrices, resulting in a matrix in $\RFM$. 
If $\xi\in {\mathcal{F}}^{(\omega)}$, the mapping $\alpha: {\mathcal{F}}^{(\omega)}\ni \xi\mapsto \alpha(\xi)=\xi\cdot A \in {\mathcal{F}}^{(\omega)}$ will be called the \emph{induced linear mapping} by $A$. In this setting, $A$ is treated as a right operator and $A=[\alpha]^{\mbox{\scriptsize \boldmath $\mathrm{E}$}}_{\mbox{\scriptsize \boldmath $\mathrm{E}$}}$, since 
\[\alpha(\mbox{\boldmath $\mathrm{e}$}_{k})=\mbox{\boldmath $\mathrm{e}$}_{k}\cdot A=(a_{k0}, a_{k1},...,a_{k\rho_{k}},0,0...)= \sum^{\rho_k}_{i=0} a_{ki}\mbox{\boldmath $\mathrm{e}$}_i.\] 

\paragraph{Change of Bases and Left Association.}
\label{sec:ChangeOfBasesAndRowEquivalence}
If a left and a right inverse of $A\in\RFM$ exist in $\RFM$, then they necessarily coincide defining the two-sided inverse of $A$ and $A$ is said to be \emph{non-singular or invertible}. In this setting, the existence of a left (resp. right) inverse entails that the induced linear mapping is surjective (resp. injective), whence the following Proposition:
\begin{proposition} \label{proposition_1.0} Let $\alpha\in \Hom$. Then $\alpha$ is an isomorphism if and only if $A=[\alpha]^{\mbox{\scriptsize \boldmath $\mathrm{E}$}}_{\mbox{\scriptsize \boldmath $\mathrm{E}$}}$ is non-singular.
\end{proposition} 
If $\mathfrak{X}\subset {\mathcal{F}}^{(\omega)}$ spans a subspace $\mathcal{Z}$ of ${\mathcal{F}}^{(\omega)}$, we shall write $\Span(\mathfrak{X})={\mathcal{Z}}$.
The \emph{null space} and the \emph{row space} of $A\in \RFM$ are defined by $\NS(A)=\Ker(\alpha)=\{\xi\in {\mathcal{F}}^{(\omega)}: \xi\cdot A=0\}$ and $\RS(A)=\Img(\alpha)=\Span(A_k)_{k\in \omega}$. The Hamel dimensions of these spaces are called \emph{nullity} and \emph{rank} of $A$, respectively, writing $\nul(A)={\rm dim}(\NS(A))$, ${\rm rank}(A)={\rm dim}(\RS(A))$. Let $f\in \Hom$ and $A=[f]^{\mbox{\scriptsize \boldmath $\mathrm{B}$}}_{{\mathcal{X}}},\; B=[f]^{\mbox{\scriptsize \boldmath $\mathrm{B}$}}_{{\mathcal{X}}^*}$, where ${\mathcal{X}^*}=\{\chi^*_k\}_{k\in\omega}$ is another basis of ${\mathcal{F}}^{(\omega)}$. Let also $\phi:{\mathcal{X}}\mapsto {\mathcal{X}^*}$ be defined by \vspace{0.02in}$\chi^*_k=\phi(\chi_k)$ and $\phi(\chi_k)=\sum_i q_{ki}\chi_i$. As the unique linear extension of the bijection $\phi$, say $\varphi:{\mathcal{F}}^{(\omega)}\mapsto {\mathcal{F}}^{(\omega)}$, is an isomorphism, the matrix $Q=(q_{ki})_{(k,i)\in \omega\times\omega}=[\varphi]^{{\mathcal{X}}}_{{\mathcal{X}}}$ is non-singular, called the \emph{matrix of passage} from the basis ${\mathcal{X}}$ to the basis ${\mathcal{X}}^*$. This passage of bases causes a change of the matrix representation of $f$ from $A$ to $B$, for a fixed basis $\mbox{\boldmath $\mathrm{B}$}$ of the codomain space, described by
\begin{equation}\label{change of basis}
Q\cdot A=B.
\end{equation}
Let $Q\in \RFM$ be non-singular such that (\ref{change of basis}) holds. Then $A,B$ is said to be \emph{left associates} or \emph{row-equivalent} and we shall write $A\sim B$. Formally ``$\sim$" is an equivalence relation on $\RFM$, which generalizes the prevailing notion of ``row-equivalence" for finite matrices. A change of the domain basis gives rise to row equivalent matrix representations of a linear mapping and visa versa.
\paragraph{Permutation Matrices and a Rearrangement Theorem.}
\label{sec:PermutationMatrices}
Given a bijection $\sigma :\omega \mapsto \omega$, the permutation matrix $P=(p_{ij})_{(i,j)\in\omega\times\omega}$ in $\RFM$ is defined by: $p_{ij}=0$, if $j\not=\sigma_i$, or $p_{i\sigma_i}=1$. Pre-multiplying $A\in \RFM$ by $P$ a permutation of the rows of $A$ occurs. Let $\sigma, \mu$ be bijections on $\omega$ and $P_{\sigma},P_{\mu}$ be the corresponding permutation matrices. Then $P_{\sigma}\cdot P_{\mu} = P_{\mu\circ\sigma}$ and $P_{\sigma^{-1}}=P^{-1}_{\sigma}$. Thus, a permutation matrix is non-singular, and its inverse is a permutation matrix too.
 
\newcounter{nroman}
Let $J=\{j_{0},j_{1},...\}$ be an infinite subset of $\omega$ such that $j_0<j_1<...$. Let also $\phi: J\mapsto \omega$ be injective and $K=\{\varphi_{j_0},\varphi_{j_1},...\}$ (the set of images of $\varphi$). The mapping $\sigma$ is defined by the following inductive process: 
\begin{list} {\roman{nroman})}{\usecounter{nroman}\setlength{\rightmargin}{\leftmargin}}
  \item  The first element, $j_0$, of $J$ is mapped to $k_0\in J$ ($k_0=\sigma(j_0)$) if and only if  $\varphi_{k_0}$ is the first element of $K$. 
	\item  The $n$-th element, $j_n$, of $J$ is mapped to $k_n\in J$ ($k_n=\sigma(j_n)$) if and only if  $\varphi_{k_n}$ is the first element of $K\setminus       \{\varphi_{k_0},  \varphi_{k_1},...,\varphi_{k_{n-1}}\}$.
\end{list}
Applying elementary arguments involving mathematical induction, we deduce that the map $J\ni j \mapsto \sigma(j)=\sigma_j\in J$ is bijective and 
the sequence $(\phi_{\sigma_i})_{i\in J}$, defined by the mapping $\phi\circ\sigma(i)$, $i\in J$, is strictly increasing. Hence, without invoking the countable axiom of choice, the following Theorem holds:
\begin{theorem} \label{Relabeling} Let $J\not=\emptyset$ and $J\subseteq \omega$. If $\phi: J\mapsto \omega$ is injective, there exists a bijection $\sigma: J \mapsto J$ such that the sequence $(\phi_{\sigma_i})_{i\in J}$ is strictly increasing.
\end{theorem}
\paragraph{Hermite Bases and Quasi-Hermite Forms.}
\label{sec:HermiteBases}
The existence and uniqueness results on QHFs of row-finite $\omega\times\omega$ matrices, attributed to Fulkerson \cite{Fu:Th}, are based on the notion of Hermite bases of subspaces of ${\mathcal{F}}^{(\omega)}$. 
\begin{definition} \label{definition-Hermite basis} 
A family $\mbox{\boldmath $\mathrm{H}$}=(\mbox{\boldmath $\mathrm{h}$}_{j})_{j\in J}$ of vectors of ${\mathcal{F}}^{(\omega)}$, $J\subseteq \omega$, is a \emph{Hermite basis} of a subspace ${\mathcal{Z}}$ of ${\mathcal{F}}^{(\omega)}$, if $\mbox{\boldmath $\mathrm{H}$}$ fulfills the subsequent conditions:
(i) $\mbox{\boldmath $\mathrm{H}$}$ spans ${\mathcal{Z}}$.
(ii) The sequence $\rho=(\rho_{j})_{j\in J}$ of rightmost index of $\mbox{\boldmath $\mathrm{H}$}$ is strictly increasing.
(iii) $h_{j\rho_{j}}=1$ for all $j\in J$ (The rightmost coefficients are ones).
(iv) If $k\in \omega$ and $j\in J$ such that $j<k$, then $h_{k \rho_{j}}=0$ (All entries in the same column and below a right leading one are zero).
\end{definition}
Let ${\mathcal{Z}}$ be an infinite dimensional subspace of ${\mathcal{F}}^{(\omega)}$. Let also $\Class_\rho$ be the set of vectors in ${\mathcal{Z}}$ of fixed length $\rho\in \Right({\mathcal{Z}})$. The infinite family $\Class=\{\Class_\rho\}_{\rho\in \Right({\mathcal{Z}})}$ partitions ${\mathcal{Z}}$ into pairwise disjoint sets such that $\displaystyle {\mathcal{Z}}=(\bigcup_{\rho\in \Right({\mathcal{Z}})}\Class_\rho)\cup\{0\}$, which thus called \emph{set of length-equivalence classes} of ${\mathcal{Z}}$. Following Fulkerson~\cite{Fu:Th}, let us choose a representative vector $\mbox{\boldmath$\mathrm{A}$}_\rho$ from each $\Class_\rho$, thus forming a sequence $\mbox{\boldmath $\mathrm{A}$}=\{\mbox{\boldmath $\mathrm{A}$}_\rho\}_{\rho\in \Right({\mathcal{Z}})}$ in ${\mathcal{Z}}$. In this respect, one formally asserts the existence of an injective choice mapping $\Phi: \Class \mapsto \cup\Class$ such that $\Phi(\Class_\rho)=\mbox{\boldmath $\mathrm{A}$}_\rho\in \Class_\rho$, that is a standard form of the axiom of countable choice. Let  $\Right({\mathcal{Z}})=\{\rho_0,\rho_1,\rho_2,...\}$ with $\rho_0<\rho_1<\rho_2<...$, and $\mbox{\boldmath$\mathrm{A}$}_{\rho_j}=(\alpha_{j0},\alpha_{j1},...,\alpha_{j\rho_{j}},0,0,...)$. 
The sequence $\mbox{\boldmath $\mathrm{H}$}=\{\mbox{\boldmath $\mathrm{H}$}_{\rho_j}\}_{j\in \omega}$ in ${\mathcal{Z}}$ generated recursively by
\begin{equation} \label{rec-Hermite}
\left\{\begin{array}{lll} 
\mbox{\boldmath $\mathrm{H}$}_{\rho_0}&=&\alpha^{-1}_{0\rho_{0}}\mbox{\boldmath $\mathrm{A}$}_{\rho_0}\vspace{0.05in}\\
\mbox{\boldmath $\mathrm{H}$}_{\rho_j}&=&
\alpha^{-1}_{j\rho_{j}}(\mbox{\boldmath $\mathrm{A}$}_{\rho_j}-\alpha_{j\rho_{j-1}}\mbox{\boldmath $\mathrm{H}$}_{\rho_{j-1}}-...-\alpha_{j\rho_{0}}\mbox{\boldmath $\mathrm{H}$}_{\rho_{0}}) \end{array} \right.
\end{equation}
turns out to be a Hermite basis of ${\mathcal{Z}}$, thus establishing the existence of this type of bases. The uniqueness of $\mbox{\boldmath $\mathrm{H}$}$ follows directly from Definition \ref{definition-Hermite basis}. In view of (\ref{rec-Hermite}), $\maxs(\mbox{\boldmath $\mathrm{H}$}_{\rho_j})=\maxs(\mbox{\boldmath $\mathrm{A}$}_{\rho_j})=\rho_j$ for all $j\in \omega$, thus
\begin{equation} \label{cond-Hermite}
\Right({\mathcal{Z}})=\Right(\mbox{\boldmath $\mathrm{A}$})=\Right(\mbox{\boldmath $\mathrm{H}$}).
\end{equation}
Relation (\ref{rec-Hermite}) is to be compared with an analogous recurrence established in~\cite{Pa:reap}. The latter recurrence generalizes an earlier result devised by Ortiz~\cite{Or:tau}\footnote{Ortiz' recursive formulation of the Lanczos' tau method was the starting landmark for a considerable amount of work in the area of polynomial approximations to differential problems, centered at Imperial college for more than twenty years. The author of this paper has had the fortune to work in this area under the supervision of Prof. Eduardo Ortiz.} generating the so-called canonical polynomials.

\begin{definition} \label{definition-Quasi-Hermite-form} \rm {$H\in \RFM$ is said to be in Quasi-Hermite form if the set of nonzero rows of $H$ is the Hermite basis of $\RS(H)$.}
\end{definition}
\begin{theorem}[Existence-Uniqueness] \label{existence-uniqueness of QHFs}
Let $A\in \RFM$. \emph{i)} There is a non-singular matrix $Q\in \RFM$ such that $Q\cdot A$ is in \emph{QHF}. \emph{ii)} $H_{1}, \ H_{2}$ are two \emph{QHFs} of $A$ if and only if there exists a permutation matrix $P\in \RFM$ such that $H_{1}=P\cdot H_{2}$.
\end{theorem}
The matrix $H:=Q\cdot A$ resulting from Theorem \ref{existence-uniqueness of QHFs} is said to be QHF of $A$.
An equivalent statement to the definition of left association, shown in~\cite{Fu:Th}, is presented below. 
\begin{proposition} \label{proposition_3.1} Two row-finite matrices in $\RFM$ are left associates \emph{(}or row equivalent\emph{)} if and only if they have the same row space and nullity. 
\end{proposition}
Using $\RS(A)$ in place of ${\mathcal{Z}}$, condition (\ref{cond-Hermite}) of recurrence (\ref{rec-Hermite}) takes the form:
\begin{equation} \label{cond-Hermite2}
\Right(\RS(A))=\Right(A).
\end{equation}
\section{Complete Bases of Right/Leftmost Index}
\label{sec:BasesOfSbspaces}
A commonly known Hamel basis of ${\mathcal{F}}^{(\omega)}$ is the set of rows of a lower triangular $\omega\times\omega$ matrix with nonzero elements 
in the diagonal, as indicated below.
\begin{proposition} \label{proposition_1.2} An indexed set $\mbox{\boldmath $\mathrm{M}$}=\{\mbox{\boldmath $\mathrm{m}$}_i\}_{i\in \omega}$ with $\mbox{\boldmath $\mathrm{m}$}_i\in {\mathcal{F}}^{(\omega)}$ defined by
\begin{equation}\label{eq-Bou1.1}	
\mbox{\boldmath $\mathrm{m}$}_i=\sum^{i}_{k=\zeta_i}  \lambda_{ik} \mbox{\boldmath $\mathrm{e}$}_{k}, 
\end{equation}
for $\lambda_{ii}\not=0$ and $0\le \zeta_i \le i$, is a Hamel basis of ${\mathcal{F}}^{(\omega)}$.
\end{proposition}
By assuming that $\lambda_{ik}=0$ for all $k$ such that $0\le k < \zeta_i$, we can identify $\zeta_i$ in (\ref{eq-Bou1.1}) with zero and we shall adopt this convention throughout this paper. In this regard, the coefficient matrix $(\lambda_{ik})_{(k,i)\in \omega\times\omega}$ in (\ref{eq-Bou1.1}) is lower triangular with nonzero elements in the diagonal. Moreover $(\lambda_{ik})_{(i,k)\in \omega\times\omega}$ indicates the matrix of passage from the basis $\mbox{\boldmath $\mathrm{E}$}$ to the basis $\mbox{\boldmath $\mathrm{M}$}$, thus being non-singular. Besides, the next statement holds (see Bourbaki~\cite[Proposition 21 pp. 218]{Bou:Alg}).
\begin{proposition} \label{proposition_1.1} Let $A\in \RFM$ and $\{A_{j}\}_{j\in J}$ be a basis of $\RS(A)$. Let also $\alpha$ be the induced linear mapping by $A$. \emph{i)} If $(z_j)_{j\in J}$ is a sequence in ${\mathcal{F}}^{(\omega)}$ such that $\alpha(z_j)=L_j$ for all $j\in J$, then $(z_j)_{j\in J}$ is a basis of a complementary space of $\NS(A)$. \emph{ii)} If $W=\omega\setminus J$, then ${\rm card} (W)= \nul (A)$.
\end{proposition}
Let $\alpha\in \Hom$ and $A=(A_{k})_{k\in \omega}=[\alpha]^{\mbox{\scriptsize \boldmath $\mathrm{E}$}}_{\mbox{\scriptsize \boldmath $\mathrm{E}$}}$. If $\{A_{k}\}_{k\in \omega}$ is a basis of ${\mathcal{F}}^{(\omega)}$, then $\alpha(\mbox{\boldmath $\mathrm{e}$}_{k})=A_k$ for all $k\in \omega$. Thus Proposition \ref{proposition_1.1} implies that $\Img(\alpha)={\mathcal{F}}^{(\omega)}$ and  $\NS(A)=\{0\}$, whence $\alpha$ is an isomorphism.
\begin{lemma} \label{lemma1}
Let $(\rho_i)_{i\in \omega}$ be a sequence in $\omega$ such that $i\le \rho_i$. Let also
\begin{equation}\label{eq-Bou1.2}	
\mbox{\boldmath $\mathrm{n}$}_i=\sum^{\rho_i}_{k=i} \lambda_{ik} \mbox{\boldmath $\mathrm{e}$}_k, 
\end{equation}
where $\lambda_{ii}\not=0$. The indexed set $\mbox{\boldmath $\mathrm{n}$}=\{\mbox{\boldmath $\mathrm{n}$}_i\}_{i\in\omega}$ with $\mbox{\boldmath $\mathrm{n}$}_i\in {\mathcal{F}}^{(\omega)}$ is linearly independent.
\end{lemma}
\begin{proof}
If $\mbox{\boldmath $\mathrm{F}$}$ is a finite subset of $\mbox{\boldmath $\mathrm{n}$}$, there exists some large enough $m\in \omega$ such that $\mbox{\boldmath $\mathrm{F}$}\subset \{\mbox{\boldmath $\mathrm{n}$}_i\}_{0\le i \le m}$. Let $\mbox{\footnotesize $\mathrm{M}$}={\rm max}(\rho_i)_{0\le i \le m}$ and $\mbox{\boldmath $\mathrm{U}$}=\{\mbox{\boldmath $\mathrm{n}$}_i\}_{0\le i \le m}\cup \{\mbox{\boldmath $\mathrm{e}$}_i\}_{m+1\le i \le \mbox{\tiny $\mathrm{M}$}}$. Let also $e$ be the standard embedding of ${\mathcal{F}}^{\mbox{\tiny $\mathrm{M}$}+1}$ into ${\mathcal{F}}^{(\omega)}$ and $e^{-1}$ the isomorphism $e^{-1}: \Img(e)\mapsto {\mathcal{F}}^{\mbox{\tiny $\mathrm{M}$}+1}$. In view of (\ref{eq-Bou1.2}), the elements of $e^{-1}(\mbox{\boldmath $\mathrm{U}$})$ form an upper triangular finite $(\mbox{\footnotesize $\mathrm{M}$}+1)\times (\mbox{\footnotesize $\mathrm{M}$}+1)$ matrix with nonzero elements in the diagonal, and so the set $e^{-1}(\mbox{\boldmath $\mathrm{U}$})$ is a basis of ${\mathcal{F}}^{\mbox{\tiny $\mathrm{M}$}+1}$. Since $e^{-1}(\mbox{\boldmath $\mathrm{F}$})\subset e^{-1}(\mbox{\boldmath $\mathrm{U}$})$, $e^{-1}(\mbox{\boldmath $\mathrm{F}$})$ is linearly independent and so is $\mbox{\boldmath $\mathrm{F}$}$. As $\mbox{\boldmath $\mathrm{F}$}$ was arbitrary, $\mbox{\boldmath $\mathrm{n}$}$ is linearly independent.
\end{proof}
However, as shown in the next counterexample, $\mbox{\boldmath $\mathrm{n}$}$ is not in general a basis of ${\mathcal{F}}^{(\omega)}$, unlike the analogous result for finite matrices and that of Proposition \ref{proposition_1.2}. Let $\mbox{\boldmath $\mathrm{a}$}=(\mbox{\boldmath $\mathrm{a}$}_i)_{i\in\omega}$ with $\mbox{\boldmath $\mathrm{a}$}_i=\mbox{\boldmath $\mathrm{e}$}_i+\mbox{\boldmath $\mathrm{e}$}_{i+1}$ for $i\in \omega$. Proposition \ref{proposition_1.2} implies that $\mbox{\boldmath $\mathrm{a}$}\cup \{\mbox{\boldmath $\mathrm{e}$}_{0}\}=\{\mbox{\boldmath $\mathrm{e}$}_{0}, \mbox{\boldmath $\mathrm{e}$}_{0}+\mbox{\boldmath $\mathrm{e}$}_{1}, \mbox{\boldmath $\mathrm{e}$}_{1}+\mbox{\boldmath $\mathrm{e}$}_{2},...\}$ is a basis of ${\mathcal{F}}^{(\omega)}$. Accordingly $\mbox{\boldmath $\mathrm{e}$}_0\not\in \Span(\mbox{\boldmath $\mathrm{a}$})$, namely $\displaystyle \Span(\mbox{\boldmath $\mathrm{a}$})\varsubsetneqq {\mathcal{F}}^{(\omega)}$. Thus $\mbox{\boldmath $\mathrm{a}$}$ is not a generating system of ${\mathcal{F}}^{(\omega)}$.
\begin{definition} \label{def-complete basis}
\rm {A Hamel basis $\mbox{\boldmath $\mathrm{F}$}$ of ${\mathcal{Z}}$, which satisfies $\Right(\mbox{\boldmath $\mathrm{F}$})=\Right({\mathcal{Z}})$ (resp. $\Left(\mbox{\boldmath $\mathrm{F}$})=\Left({\mathcal{Z}})$), will be referred to as a \emph{complete basis of rightmost index} (resp. \emph{leftmost index}) of ${\mathcal{Z}}$. If ${\mathcal{Z}}$ is the row space of a matrix, then we shall equivalently use the terms ``\emph{row-length}" and ``rightmost index}".
\end{definition}
In view of (\ref{cond-Hermite}), the bases $\mbox{\boldmath $\mathrm{H}$}$ and $\mbox{\boldmath $\mathrm{A}$}$ whose terms are connected by (\ref{rec-Hermite}) are complete bases of rightmost index. Two useful complete bases of rightmost index (resp. leftmost index) of subspaces of ${\mathcal{F}}^{(\omega)}$ are given in what follows.

\begin{proposition} \label{proposition_1.3} Let $\mbox{\boldmath $\mathrm{F}$}=(\mbox{\boldmath $\mathrm{F}$}_{j})_{j\in J}$ be a family of nonzero vectors in ${\mathcal{F}}^{(\omega)}$. Let also ${\mathcal{Z}}=\Span(\mbox{\boldmath $\mathrm{F}$}_j)_{j\in J}$.
\emph{i)} If the sequence $(\rho_{j})_{j\in J}$ of rightmost index of $\mbox{\boldmath $\mathrm{F}$}$ is injective, then 
$\mbox{\boldmath $\mathrm{F}$}$ is a complete basis of rightmost index of ${\mathcal{Z}}$.
\emph{ii)} If the sequence $(\zeta_{j})_{j\in J}$ of leftmost index of $\mbox{\boldmath $\mathrm{F}$}$ is injective, then 
$\mbox{\boldmath $\mathrm{F}$}$ is a complete basis of leftmost index of ${\mathcal{Z}}$.
\end{proposition}
\begin{proof} i) Since $\Right(\mbox{\boldmath $\mathrm{F}$})\subset \Right({\mathcal{Z}})$, we are to show that $\Right({\mathcal{Z}})\subset \Right(\mbox{\boldmath $\mathrm{F}$})$. Let $i\in \Right({\mathcal{Z}})$. By definition there exists some $Z\in {\mathcal{Z}}$ such that $\maxs(Z)=i$. As $\mbox{\boldmath $\mathrm{F}$}$ spans ${\mathcal{Z}}$, we conclude that $Z=\lambda_{j_{1}}\mbox{\boldmath $\mathrm{F}$}_{j_{1}}+\lambda_{j_{2}}\mbox{\boldmath $\mathrm{F}$}_{j_{2}}+...+\lambda_{j_{n}}\mbox{\boldmath $\mathrm{F}$}_{j_{n}}$. Since $\rho$ is injective, we may assume that $\rho_{j_{1}}<\rho_{j_{2}}<...<\rho_{j_{n}}$. Thus $i=\rho_{j_{n}}\in \Right(\mbox{\boldmath $\mathrm{F}$})$, as required. In showing that $\mbox{\boldmath $\mathrm{F}$}$ is basis of ${\mathcal{Z}}$, it remains to show that $\mbox{\boldmath $\mathrm{F}$}$ is linearly independent. Whereas $\rho:J\mapsto \omega$ is injective and $\Right({\mathcal{Z}})=\Right(\mbox{\boldmath $\mathrm{F}$})$, it follows that $\rho:J\mapsto \Right({\mathcal{Z}})$ is bijective. Accordingly, the terms of $\mbox{\boldmath $\mathrm{F}$}$ can be relabeled by setting $\mbox{\boldmath $\mathrm{F}$}'_{\rho_{j}}=\mbox{\boldmath $\mathrm{F}$}_{j}$ for all $j\in J$. Since $\{\mbox{\boldmath $\mathrm{F}$}'_{\rho_j}\}_{j\in J}$ and $\{\mbox{\boldmath $\mathrm{F}$}'_i\}_{i\in \Right({\mathcal{Z}})}$ coincide as sets, the terms of $\displaystyle {\mathcal{X}}=\{\mbox{\boldmath $\mathrm{F}$}'_i\}_{i\in \Right({\mathcal{Z}})}\cup \{\mbox{\boldmath $\mathrm{e}$}_i\}_{i\in\omega\setminus \Right({\mathcal{Z}})}$ satisfy (\ref{eq-Bou1.1}). Proposition \ref{proposition_1.2} implies that ${\mathcal{X}}$ is a basis of ${\mathcal{F}}^{(\omega)}$. Since  $\mbox{\boldmath $\mathrm{F}$}=\{\mbox{\boldmath $\mathrm{F}$}'_i\}_{i\in \Right({\mathcal{Z}})}\subset {\mathcal{X}}$ the assertion follows. ii) By substituting $\zeta$ for $\rho$ and using the above arguments, we conclude that $\Left ({\mathcal{Z}})=\Left(\mbox{\boldmath $\mathrm{F}$})$. Whereas $\zeta: J\mapsto \Left(\mbox{\boldmath $\mathrm{F}$})$ is bijective, we define $\mbox{\boldmath $\mathrm{F}$}'_{\zeta_j}=\mbox{\boldmath $\mathrm{F}$}_j$ for $j\in J$ and $\mbox{\boldmath $\mathrm{F}$}=\{\mbox{\boldmath $\mathrm{F}$}'_{\zeta_j}\}_{j\in J}=\{\mbox{\boldmath $\mathrm{F}$}'_m\}_{m\in \Left(\mbox{\footnotesize\boldmath $\mathrm{F}$})}$. The set $\{\mbox{\boldmath $\mathrm{F}$}'_m\}_{m\in \Left(\mbox{\footnotesize\boldmath $\mathrm{F}$})}\cup\{\mbox{\boldmath $\mathrm{e}$}_m\}_{m\in \omega\setminus \Left(\mbox{\footnotesize\boldmath $\mathrm{F}$})}$ is linearly independent, as satisfying (\ref{eq-Bou1.2}), and so is its subset $\mbox{\boldmath $\mathrm{F}$}$, as required. 
\end{proof}
\begin{corollary} \label{corollary_1.4} Let $A\in \RFM$. Let also the sequence of rightmost index \emph{(}resp. leftmost index\emph{)} of the nonzero rows, say $(A_{j})_{j\in J}$, of $A$ be injective. Then:
\emph{i)} The indexed set $(A_{j})_{j\in J}$ is a complete basis of rightmost (resp. leftmost index) of $\RS(A)$. 
\emph{ii)} ${\rm card} (W)= \nul (A)$ \emph{(}$W=\omega\setminus J$\emph{)}.
\end{corollary}
\begin{proof} Proposition \ref{proposition_1.3} applied with $\mbox{\boldmath $\mathrm{F}$}=(A_{j})_{j\in J}$ and ${\mathcal{Z}}=\RS(A)$ imply directly statements (i).  Statement (ii) follows from Proposition \ref{proposition_1.1} (ii). 
\end{proof}
An equivalent statement to Definition \ref{def-complete basis}, concerning complete bases of rightmost index exclusively, is given below.
\begin{theorem} \label{theorem complete-basis} Let ${\mathcal{Z}}$ be a subspace of ${\mathcal{F}}^{(\omega)}$. Then $\mbox{\boldmath $\mathrm{F}$}=(\mbox{\boldmath $\mathrm{F}$}_j)_{j\in J}$ is a complete basis of rightmost index of ${\mathcal{Z}}$ if and only if ${\mathcal{Z}}=\Span(\mbox{\boldmath $\mathrm{F}$}_j)_{j\in J}$ and the sequence $(\rho_{j})_{j\in J}$ of rightmost index of $\mbox{\boldmath $\mathrm{F}$}$ is injective.
\begin{proof} Let $\mbox{\boldmath $\mathrm{F}$}_j=(f_{jk})_{k\in \omega}$. On the contrary, let there are $j_0, j_1$ in $J$ such that $j_0\not= j_1$ and $\rho_{j_0}=\rho_{j_1}$. Let also $m_1:=\rho_{j_1}$ and $\lambda_0:=f_{j_1 m_1}, \lambda_1:=f_{j_0 m_1}$ (in that order). Certainly $\lambda_0\not= 0, \ \lambda_1\not= 0$ and $\mbox{\boldmath $\mathrm{F}$}_{j_0}=(f_{j_0 0},f_{j_0 1},..., \lambda_1,0,...)$, $\mbox{\boldmath $\mathrm{F}$}_{j_1}=(f_{j_1 0},f_{j_1 1},..., \lambda_0,0,...)$, where $\lambda_0, \lambda_1$ have as column position $m_1$. Thus $\maxs (\lambda_0\mbox{\boldmath $\mathrm{F}$}_{j_0}-\lambda_1 \mbox{\boldmath $\mathrm{F}$}_{j_1})<m_1$ and $\lambda_0\mbox{\boldmath $\mathrm{F}$}_{j_0}-\lambda_1 \mbox{\boldmath $\mathrm{F}$}_{j_1}\not=\textbf{0}$, since $\mbox{\boldmath $\mathrm{F}$}_{j_0}, \mbox{\boldmath $\mathrm{F}$}_{j_1}$ are linearly independent. Let us call $m_2=\maxs (\lambda_0\mbox{\boldmath $\mathrm{F}$}_{j_0}-\lambda_1 \mbox{\boldmath $\mathrm{F}$}_{j_1})$. As $\lambda_0\mbox{\boldmath $\mathrm{F}$}_{j_0}-\lambda_1 \mbox{\boldmath $\mathrm{F}$}_{j_1} \in {\mathcal{Z}}$ and $\Right({\mathcal{Z}})=\Right(\mbox{\boldmath $\mathrm{F}$})$, we infer $m_2\in \Right(\mbox{\boldmath $\mathrm{F}$})$. Thus, there is some $\mbox{\boldmath $\mathrm{F}$}_{j_2}$ in $\mbox{\boldmath $\mathrm{F}$}$ such that $\maxs (\mbox{\boldmath $\mathrm{F}$}_{j_2})=m_2$ and $m_2<m_1$. 
As $m_2=\maxs (\lambda_0\mbox{\boldmath $\mathrm{F}$}_{j_0}-\lambda_1 \mbox{\boldmath $\mathrm{F}$}_{j_1})=\maxs (\mbox{\boldmath $\mathrm{F}$}_{j_2})$, there is $\lambda_2\not=0$ such that $m_3=\maxs(\lambda_0\mbox{\boldmath $\mathrm{F}$}_{j_0}-\lambda_1 \mbox{\boldmath $\mathrm{F}$}_{j_1}-\lambda_2 \mbox{\boldmath $\mathrm{F}$}_{j_2})$ and $m_3<m_2$. Now $\lambda_0\mbox{\boldmath $\mathrm{F}$}_{j_0}-\lambda_1 \mbox{\boldmath $\mathrm{F}$}_{j_1}-\lambda_2 \mbox{\boldmath $\mathrm{F}$}_{j_2}\not=\textbf{0}$, since $\mbox{\boldmath $\mathrm{F}$}_{j_0},\mbox{\boldmath $\mathrm{F}$}_{j_1},\mbox{\boldmath $\mathrm{F}$}_{j_2}$ are linearly independent, thus $m_3\in \Right(\mbox{\boldmath $\mathrm{F}$})$. This process defines a strictly decreasing infinite sequence $m=(m_i)$ in $\naturals$, which is impossible. The converse statement is equivalent to Proposition \ref{proposition_1.3}.
\end{proof}
\end{theorem}
Unlike Theorem \ref{theorem complete-basis}, the sequence $(\zeta_{j})_{j\in J}$ of a complete basis of leftmost index is not necessarily injective. To see this, let us recall the basis $\mbox{\boldmath $\mathrm{a}$}\cup \{\mbox{\boldmath $\mathrm{e}$}_{0}\}$ of ${\mathcal{F}}^{(\omega)}$, as defined in the previous counterexample. Call this basis $\mbox{\boldmath $\mathrm{F}$}$. As $\Left(\mbox{\boldmath $\mathrm{F}$})=\omega$, it follows that $\Left({\mathcal{F}}^{(\omega)})=\Left(\mbox{\boldmath $\mathrm{F}$})$. Thus $\mbox{\boldmath $\mathrm{F}$}$ is a complete basis of leftmost index of ${\mathcal{F}}^{(\omega)}$ but the sequence $(\zeta_{j})_{j\in J}$ is not injective, since $\zeta_0=\zeta_1=0$.
\begin{corollary} \label{Corollary_1.5} Let $\mbox{\boldmath $\mathrm{F}$}=(\mbox{\boldmath $\mathrm{F}$}_j)_{j\in J}$ be a complete basis of rightmost index of ${\mathcal{Z}}$. \emph{(i)} The set $\mbox{\boldmath $\mathrm{F}$}\cup\{\mbox{\boldmath $\mathrm{e}$}_i\}_{i\in \omega\setminus\Right(\mbox{\footnotesize\boldmath $ \mathrm{F}$})}$ is a complete basis of rightmost index of ${\mathcal{F}}^{(\omega)}$.\\
\emph{(ii)} The set $\{\mbox{\boldmath $\mathrm{e}$}_i\}_{i\in \omega\setminus\Right(\mbox{\footnotesize\boldmath $ \mathrm{F}$})}$ is a basis of a complementary space of ${\mathcal{Z}}$.
\end{corollary}
\begin{proof} (i) Theorem \ref{theorem complete-basis} entails that $\rho: J\mapsto \Right(\mbox{\boldmath $\mathrm{F}$})$ is bijective. Thus $\mbox{\boldmath $\mathrm{F}$}$ can be relabeled by setting $\mbox{\boldmath $\mathrm{F}$}'_{\rho_{j}}=\mbox{\boldmath $\mathrm{F}$}_{j}$ for all $j\in J$ and $\mbox{\boldmath $\mathrm{F}$}'=\{\mbox{\boldmath $\mathrm{F}$}'_i\}_{i\in \Right(\mbox{\footnotesize\boldmath $ \mathrm{F}$})}$. Let ${\mathcal{X}}=\mbox{\boldmath $\mathrm{F}$}'\cup \{\mbox{\boldmath $\mathrm{e}$}_i\}_{i\in \omega\setminus\Right(\mbox{\footnotesize\boldmath $ \mathrm{F}$})}$. 
Since $\mbox{\boldmath $\mathrm{F}$}'$ and $\mbox{\boldmath $\mathrm{F}$}$ coincide as sets, using the same argument as in the proof of Proposition \ref{proposition_1.3} (i), we deduce that ${\mathcal{X}}$ is a complete basis of rightmost index of ${\mathcal{F}}^{(\omega)}$, as required. Statement (ii) follows directly from (i).
\end{proof}
\section{Lower Row-Reduced and Row-Echelon Forms}
\label{sec:RowFiniteMatricesInLRRF}
The terms ``lower" and ``upper" readily characterize matrices in echelon form, but not directly matrices in row-reduced form. However, if $A\in \RFM$ is in LRRF, there is a permutation matrix, $P$, such that  $P\cdot A$ is in LREF (see Proposition \ref{proposition LRRF_3}). A similar result holds for matrices in URRF (see Proposition \ref{proposition URRF_1}). In this regard, with an abuse of language, the terms ``lower" and ``upper" are also used to distinguish $\omega\times\omega$ matrices in row-reduced form. These forms of row-finite matrices generalize analogous results on finite matrices.
\begin{definition} \label{definition-LRRF} {\rm Let $A\in \RFM$. Let also $(A_{j})_{j\in J}$ be the sequence of nonzero rows of  $A=(a_{ki})_{(k,i)\in \omega\times\omega}$. Then $A$ is said to be in \emph{lower row-reduced form} if the following conditions are satisfied: 
i) $a_{j\rho_{j}}=1$ for all $j\in J$ (Right leading coefficients are ones).
ii) If $j\in J$ and $k\in \omega$ such that $k\not=j$, then $a_{k\rho_{j}}=0$ (A column containing a right leading one has zeros everywhere else).}
\end{definition}
In order to cover the case of row-finite $\omega\times\omega$ matrices of infinite nullity, the standard condition on finite matrices in echelon form which requires that zero rows are to be grouped together, is not included in the subsequent definition. 
\begin{definition} \label{definition-LREF} {\rm  The matrix $A\in \RFM$ is in \emph{lower row-echelon form} (LREF) if the sequence $(\rho_j)_{j\in \omega}$ of row-length of $A$ is strictly increasing.}
\end{definition}
Substituting leftmost for rightmost coefficients in Definitions \ref{definition-LRRF} and \ref{definition-LREF}, row-finite $\omega\times\omega$ matrices in URRF and in UREF are formally defined; in accord with the prevailing definitions of finite matrices in row-reduced and row-echelon form. If $A$ is in LRRF (resp. LREF) and $A\sim B$, then $A$ is said to be a LRRF (resp. LREF) of $B$. Inasmuch as row permutations on a matrix in LRRF do not affect Definition \ref{definition-LRRF}, the following Proposition is a direct consequence.
\begin{proposition} \label{proposition LRRF_0} Let $P$ be a permutation $\omega\times\omega$ matrix. If $A\in \RFM$ is in \emph{LRRF}, then $P\cdot A$ is in \emph{LRRF} and $\Right(A)=\Right(P\cdot A)$. 
\end{proposition}
\begin{proposition} \label{proposition LRRF_1} Let $A=(a_{ki})_{(k,i)\in \omega\times\omega}$ be in \emph{LRRF}. Then the map $\rho:J\mapsto \Right(A)$ of row-length of $A$ is bijective.
\end{proposition}
\begin{proof} By definition $\rho$ is surjective. Let $k,j\in J$ such that $k\not=j$ and $\rho_{k}=\rho_{j}=m$. Definition \ref{definition-LRRF}(i) implies that $a_{km}=a_{k \rho_{k}}=1$. As $k\not=j$, Definition \ref{definition-LRRF}(ii) implies that $a_{km}=a_{k \rho_{j}}=0$, which is impossible. Thus $\rho$ is also injective.
\end{proof}
Since row-finite matrices in LRRF and LREF are associated with injective maps of row-length, Corollary \ref{corollary_1.4} entails the following statement.
\begin{proposition} \label{proposition LREF_1} Let $A\in\RFM$ be in \emph{LRRF} or in \emph{LREF}. \emph{i)} $\Right(A)=\Right(\RS(A))$ and $\rho: J\mapsto \Right(\RS(A))$ is bijective. \emph{ii)} The set of nonzero rows of $A$ is a complete basis of  row-length of $\RS(A)$. \emph{iii)} ${\rm card}(W)=\nul (A)$.
\end{proposition}
As shown below, the existence of a row equivalent LRRF of every matrix in $\RFM$ can be deduced from Theorem \ref{existence-uniqueness of QHFs}(i).
\begin{theorem}[Existence] \label{QH is in LRRF} A row-finite $\omega\times\omega$ matrix $H$ in \emph{QHF} is simultaneously in \emph{LRRF} and in \emph{LREF}. Conversely, a row-finite $\omega\times\omega$ matrix in \emph{LRRF} and in \emph{LREF} is in \emph{QHF}.
\end{theorem}
\begin{proof} As the sequence of row-length $(\rho_{j})_{j\in J}$ of $H$ is strictly increasing, it follows that $k<j \Rightarrow \rho_{k}<\rho_{j}$. The latter inequality implies that $a_{k \rho_{j}}$ is positioned further to the right than the rightmost one of the $k$-th row and so $a_{k \rho_{j}}=0$ for all $k\in J$ with $k<j$. Also, $a_{k \rho_{j}}=0$ for all $k\in W$. In view of Definition \ref{definition-Hermite basis}(iv), the condition (ii) of Definition \ref{definition-LRRF} follows. As the condition (i) of Definition \ref{definition-LRRF} is trivially satisfied, the direct implication follows. The converse implication follows directly from the definitions.
\end{proof}
\begin{proposition} \label{proposition LRRF_3}
Let $A\in \RFM$ and $A\not=\mbox{\boldmath $\mathrm{0}$}$. If the map $\rho:J\mapsto \Right(A)$ of row-length of $A$ is injective, there is a permutation $\omega\times\omega$ matrix $P$ such that $P\cdot A$ is in \emph{LREF}.
\end{proposition}
\begin{proof} 
It follows from $A\not=\mbox{\boldmath $\mathrm{0}$}$ that $J\not=\emptyset$. Moreover, as $\rho$ is injective, Theorem \ref{Relabeling} implies the existence of a bijection $\sigma:J\mapsto J$ such that $(\rho_{\sigma_j})_{i\in J}$ is strictly increasing. Let 
\begin{equation} \label{bijection}
\mu_i=\left\{\begin{array}{cc} \sigma_i,  & \ {\rm if} \ i\in J \\
                                         i,   &  \ {\rm if} \  i\in W    
\end{array}\right. 
\end{equation}
On account of $J\cap W=\emptyset$ and $J\cup W=\omega$, as $\mu_{j}=\sigma_{j}\in J$ for all $j\in J$ and $\mu:W \mapsto W$ is the identity, we conclude that $\mu: \omega\mapsto \omega$ is bijective. Thus $\mu$ defines a permutation $\omega\times\omega$ matrix $P$ such that $\maxs(P\cdot A)_{\mu_j}=\rho_{\sigma_j}$, $j\in J$ and so the sequence of row-length of $P\cdot A$ is strictly increasing, as claimed. 
\end{proof}
\begin{corollary} \label{corollary LRRF_4}
If $A\in \RFM$ is in \emph{LRRF}, there is a permutation matrix $P$ such that $P\cdot A$ is both in \emph{LRRF} and \emph{LREF} \emph{(}or in \emph{QHF}\emph{)} and $P\cdot A\sim A$.
\end{corollary}
\begin{proof} Proposition \ref{proposition LRRF_1} implies that the map of row-length of $A$ is injective and so Proposition \ref{proposition LRRF_3} entails that there exists a permutation matrix $P$ such that $P\cdot A$ is in LREF. Proposition \ref{proposition LRRF_0} entails that $P\cdot A$ is in LRRF too, and $A\sim P\cdot A$, since $P$ is non-singular. Thus Theorem \ref{QH is in LRRF} implies the assertion. 
\end{proof}
\begin{theorem}[Uniqueness] \label{uniqueness of LRRF} Let $A, B\in \RFM$ be in \emph{LRRF} and $A\sim B$. There exists a permutation matrix $P$ such that $A=P\cdot B$.
\end{theorem}
\begin{proof} In view of Corollary \ref{corollary LRRF_4}, there exist permutation matrices $R,S$ such that $Q_1=R\cdot A$ and $Q_2=S\cdot B$, where $Q_1, \ Q_2$ are QHFs  of $A$, $B$, respectively. As $Q_1\sim A\sim B \sim Q_2$, Theorem \ref{existence-uniqueness of QHFs} implies that $Q_1=V\cdot Q_2$ for some permutation matrix $V$. Thus $A=R^{-1}\cdot Q_1=R^{-1}\cdot V\cdot Q_2=R^{-1}\cdot V\cdot S\cdot B$. Since $P=R^{-1}\cdot V\cdot S$ is a permutation matrix, the assertion follows.
\end{proof}
The above statement extends Theorem \ref{existence-uniqueness of QHFs} (ii) to matrices in LRRF. It entails that two row equivalent row-finite $\omega\times\omega$ matrices in LRRF may be different sequences of row-vectors but they are identical as sets of row-vectors. This justifies the ``almost" uniqueness characterization of matrices in LRRF. 
\section{Infinite Gauss-Jordan Elimination Algorithm}
\label{sec:TheInfiniteGauss-JordanEliminationAlgorithm}
The infinite Gauss-Jordan algorithm is primarily designed to generate a sequence of matrices in LRRF, excluding row permutations, in order to meet the theoretical purposes of this paper (see Section \ref{sec:TheInfiniteGJEliminationAndTheRowEquivalence}). At this stage of analysis, no qualitative distinction is made between leftmost and rightmost pivot strategies (LPS and RPS), since both strategies equally work on finite matrices, while the resulting matrix, upon algorithm completion, is not under consideration yet.
\begin{definition}
A \emph{top submatrix} of $A\in \RFM$ of order $n$ is the finite sequence of the first $n+1$ rows of $A$ denoted by $A\!\!\mid_n=(A_k)_{0\le k \le n}$.
\end{definition}
\textbf{The Algorithm:} Let $A\in \RFM$ have at least two nonzero rows. Otherwise, $A$ consists of at most one nonzero row, thus being both in LRRF and in URRF and the process ends. Let $A_{n_0}, A_{n_1}$ be the first and second nonzero rows of $A$ respectively. Starting with the top submatrix $A\!\!\mid_{n_1}$, the rightmost (resp. leftmost) nonzero entry of $A_{n_0}$ is used as a pivot to clear the corresponding entry (entry in the same column) of $A_{n_1}$  (elimination elementary operation). Let $G_{n_1}$ be the resulting row. If $G_{n_1}\not=\textbf{0}$, the rightmost (resp. leftmost) nonzero entry of $G_{n_1}$ is used as a pivot to clear the corresponding entry of $A_{n_0}$. Otherwise, the algorithm leaves a zero row and continues to the next step. Leading coefficients are normalized to one (scaling elementary operation) and the resulting matrix is a LRRF (resp. URRF) of $A\!\!\mid_{n_1}$. The reduced matrix is augmented by successive rows of the remaining rows of $A$ up to and including the first nonzero row encountered. Let $A_{n_i}$ be a new included row. The \emph{Gaussian elimination} uses the nonzero rows of the previously reduced matrix as pivot rows to clear corresponding entries of $A_{n_i}$ resulting in $G_{n_i}$. If $G_{n_i}\not=\textbf{0}$, the \emph{Jordan elimination} uses the leading coefficient of $G_{n_i}$ as a pivot to clear the entries in the column above this pivot. Otherwise, the algorithm leaves a zero row and continues to the next step. Upon completion of \emph{normalization}, the $n_i$ stage of the process ends with a matrix in LRRF (resp. URRF). This sequential process continues in this manner ad infinitum.

Each stage of the algorithmic process, as described above, is divided in three parts: Gaussian elimination, Jordan elimination and normalization. 
\begin{remark} \rm{ In the standard version of the Gaussian elimination, row permutations are used to find a pivot to avoid a break down, whenever a zero appears as a candidate for a pivot. However, the approach adopted here avoids row permutations in searching for a pivot, by using as pivots the rightmost (resp. leftmost) coefficients of nonzero rows already created in previous stages. Evidently, the algorithm presented in this Section also works on finite matrices. For computational and programming needs, we can also work with finite augmented matrices by defining as column dimension of each augmented matrix the integer $\mbox{\scriptsize $\mathrm{N}$}_i= \max\{\mbox{\scriptsize $\mathrm{N}$}_{i-1}, \maxs (A_{n_i})\}$, starting with $\mbox{\scriptsize $\mathrm{N}$}_1= \max\{\maxs (A_{n_0}), \maxs (A_{n_1})\}$.} 
\end{remark}
\section{LRRFs versus URRFs and RPS versus LPS}
\label{sec:LeftVersusRightPivoting} 
An intrinsic defect of row-finite matrices in URRF is analyzed in the present Section in connection with the infinite Gauss-Jordan elimination algorithm implemented with LPS. Example \ref{example1} shows that such defect is not due to an inherent weakness in the algorithm itself, but rather to the essence of row-finite $\omega\times\omega$ matrices in URRF, which do not preserve row-equivalence. 
\begin{proposition} \label{proposition URRF_1} Let $A\in \RFM$ be in \emph{URRF}. Then there is a permutation matrix $P$ such that $P\cdot A$ is in \emph{UREF} and in \emph{URRF} and $\Left(A)=\Left(P\cdot A)$.
\end{proposition} 
\begin{proof} We can rephrase Propositions \ref{proposition LRRF_0}, \ref{proposition LRRF_1}, \ref{proposition LRRF_3} by using URRF and UREF in place of LRRF and LREF respectively. Then proceeding as in the proof of Corollary \ref{corollary LRRF_4}, the assertion follows.
\end{proof}
\begin{proposition} \label{proposition UREF_1} If $B,C\in \RFM$ are both in \emph{UREF} and $B\sim C$, then $\Left(B)=\Left(C)$. 
\end{proposition} 
\begin{proof} Corollary \ref{corollary_1.4} implies that $\Left(B)=\Left(\RS(B))$ and $\Left(C)=\Left(\RS(C))$. As $B\sim C$, it follows that $\RS(B)=\RS(C)$, whence $\Left(B)=\Left(\RS(B))=\Left(\RS(C))=\Left(C)$, as claimed.
\end{proof}
The following Example concludes the discussion of this Section. 
\begin{example}[Counterexample] \label{example1}
{\rm Let us consider the row-finite and column-finite $\omega\times\omega$ matrix:
\begin{equation} \label{ex1.matr2}
A= \left( \begin{array}{cccccc} 
1 & 1 & 0 & 0 & 0 &...\\
0 & 1 & 1 & 0 & 0 &...\\
0 & 0 & 1 & 1 & 0 &...\vspace{-0.05in}\\
. & . & . & . & . &...
\end{array}   \right) 
\end{equation} 
The infinite Gauss-Jordan algorithm is implemented with LPS to row reduce $A$. After including the $n$-th row of $A$, the resulting matrix in URRF is of the form:
\begin{equation} \label{eq-ex1.matrU}
U_n=\left( \begin{array}{llllccc}
1 & 0  &...& 0 & (-1)^{n-1} \vspace{0.05in}& 0 &... \\
0 & 1  &...& 0 & (-1)^{n-2}& 0 &... \vspace{-0.05in}\\
.&.&...&.&.&.&... \\
0 & 0  & ... & 1 & 1& 0 &...\\
\end{array}
\right)
\end{equation}
We observe that, as the process progresses, the rightmost nonzero column moves further to the right. The latter entails that $U_n$ does not contain the previous reduced matrix $U_{n-1}$ as a submatrix and the length of each row increases indefinitely. This is due to the fact that, the elimination of an entry causes the simultaneous generation of a new nonzero entry further to the right in the same row, which, in turn, is vanished by a successor pivot element and so ad infinitum. For instance, two consecutive forms of the first row are given by
\[\left\{\begin{array}{ccccc} 
1,\ 0,\ 0,...,\ 0,& (-1)^{n-1} &, &  0 &, 0,...)   \vspace{0.05in}\\
1,\ 0,\ 0,...,\ 0,&                  0  &,       &(-1)^{n}&, 0,...)
\end{array}
\right.
\]
Thus we must accept either that the process does not end up with a certain outcome or that upon algorithm completion all the coefficients above rightmost nonzero entries are ultimately vanished, thus reaching the identity $\omega\times\omega$ matrix $\textbf{I}$.
However, if one follows the latter view, then must also accept that every row in $\textbf{I}$ is the result of an infinite sequence of elementary row operations. It amounts to the same to say that every row of $\textbf{I}$ must be an infinite linear combination of nonzero rows of $A$. But this event challenges row equivalence. Using the statements established earlier in this paper a formal answer to this challenge is given in what follows. Since $A$ is in LREF, Proposition \ref{proposition LREF_1} implies that $\Right(A)=\Right(\RS(A))$. On account of $\Right(A)=\{1,2,...\}$, it follows that $0\not\in \Right(\RS(A))$, whence $\mbox{\boldmath $\mathrm{e}$}_0\not\in \RS(A)$. Consequently $\RS(A)\varsubsetneqq {\mathcal{F}}^{(\omega)}=\RS(\textbf{I})$ and so Proposition \ref{proposition_3.1} 
entails that $A\not\sim \textbf{I}$. The latter is to be compared with the relevant result showing that the indexed set $\mbox{\boldmath $\mathrm{n}$}$ in Lemma \ref{lemma1} does not span ${\mathcal{F}}^{(\omega)}$.

It is shown next that the only possible URRFs of $A$ in (\ref{ex1.matr2}) are row permutations of $\textbf{I}$. Let $U$ be in URRF and $U\sim A$. Proposition \ref{proposition URRF_1} implies that there is a permutation matrix, $P$, such that $P\cdot U$ is in UREF and in URRF. As $A$ is in LREF and the indexing set of zero rows of $A$ is empty, Proposition \ref{proposition LREF_1} implies $\nul(A)=0$. Let $U^*=P\cdot U$ and $\zeta: J\mapsto \Left(U^*)$ be the map defining the sequence of 
leftmost index of $U^*$. Certainly $W=\omega\setminus J$ is the indexing set of zero rows of $U^*$. Since $\nul(A)=0$ and $A\sim U^*$, it follows from Proposition \ref{proposition_3.1} that $\nul(U^*)=0$. Formally $\card(W)\le \nul(U^*)=0$, whence $W=\emptyset$ and $J=\omega$. Whereas $A,U^*$ are in UREF and $A\sim U^*$, Proposition \ref{proposition UREF_1} implies $\Left(U^*)=\Left(A)=\omega$. As $\zeta:\omega\mapsto \omega$ is both surjective and strictly increasing, it follows that $\zeta$ is an order isomorphism on $\omega$. But the only order isomorphism from $\omega$ to itself is the identity. Hence the leftmost ones of $U^*$ are the elements of its main diagonal. Since $U^*$ is in URRF, the elements above and below the main diagonal of $U^*$ are zero and so $U^*=\textbf{I}$. Taking into account that $P^{-1}$ is a permutation matrix, the assertion follows from $P\cdot U=\textbf{I} \Leftrightarrow U=P^{-1}\cdot \textbf{I}$. Consequently $\RS(A)\varsubsetneqq \RS(\textbf{I})=\RS(U)$ and so $A\not\sim U$. 

The same arguments show the non-existence of URRF for an extensive class of non-invertible row-finite matrices such as three-diagonal infinite matrices.

Let us next apply the infinite Gauss-Jordan algorithm implemented with RPS to $A$ in (\ref{ex1.matr2}). The resulting matrix is in LRRF, as displayed below:
\begin{equation} \label{eq-ex2.matr2}
H=\left(
\begin{array}{rllllll}
 1 & 1 & 0 & 0 &...& 0 &... \vspace{0.05in}\\
 -1 & 0 & 1 & 0 &...& 0 &...\vspace{0.06in}\\
1 & 0 & 0 & 1 & ... & 0 &...\vspace{-0.05in}\\ 
 .&.&.&.&...&.&...  
\end{array}
\right) 
\end{equation}
By applying the same sequence of row operations to the identity $\omega\times\omega$ matrix $\textbf{I}$, the derived matrix, say $Q$, is given by 
\[ Q=\left( \begin{array}{rrcccc} 
1&0&0&0&...\vspace{0.05in}\\
-1&1&0&0&...\vspace{0.05in}\\
1&-1&1&0&...\vspace{-0.05in}\\
.&.&.&.&...
      \end{array}  \right)
\]
Formally $Q$ is non-singular, as being lower triangular with non-zero diagonal elements. Since $H=Q\cdot A$, the row-equivalence of $H$ and $A$ follows.} 
\end{example}
\section{The Infinite Gauss-Jordan Algorithm Implemented with RPS and the Row-Equivalence}
\label{sec:TheInfiniteGJEliminationAndTheRowEquivalence}
The main goal of this paper is established in the current Section. It is shown that the infinite Gauss-Jordan algorithm implemented with RPS ends up as a LRRF form of a row-finite $\omega\times \omega$ matrix respecting row equivalence. As a consequence, the constructiveness of a QHF is shown at the end of this Section.

From here onwards, $C\in \RFM$ and $C_n$ denotes the row of index $n$ of $C$. $L^{(i)}_n$ stands for the row of index $n$, after the application of the row of index $i$ to the previously obtained $L^{(i-1)}_n$, either by the Gaussian ($i<n$) or the Jordan elimination ($i>n$).  
The entries of \vspace{0.02in}$L^{(i)}_n$ will be denoted by $l^{(i)}_{nj}$. In the Gaussian elimination, the rows $L^{(n-1)}_i$, $0\le i\le n-1$, are successively applied to the row of index $n$, starting with the application of $L^{(n-1)}_0$ to $C_n$ resulting in $L^{(0)}_n$. Then the row $L^{(n-1)}_{1}$ is applied to $L^{(0)}_n$ resulting in $L^{(1)}_n$ and finally the row $L^{(n-1)}_{n-1}$ is applied to $L^{(n-2)}_n$ resulting in $L^{(n-1)}_n$. Aimed at normalizing the leading coefficient to one, let us define
\begin{equation} \label{leading}
G_n=\left\{\begin{array}{cl} \frac{1}{l^{(n-1)}_{n\rho_n}}L^{(n-1)}_n  & {\rm if} \  L^{(n-1)}_n\not=\textbf{0}\vspace{0.05in}\\
                               \textbf{0} &   {\rm otherwise}   
                               \end{array}\right.   
\end{equation}
where $l^{(n-1)}_{n\rho_n}$ denotes the right leading coefficient of $L^{(n-1)}_{n}$. In order to unify formulas, we shall further adopt the notation $L^{(-1)}_{n}=C_n$ and $G_n=L^{(n)}_{n}$, since \vspace{0.02in} these symbols were not previously used. 
By analogy, the entries of $G_{n}$ will be denoted by $g_{nj}$, whence \vspace{0.02in}$g_{nj}=l^{(n)}_{nj}$. 
The Jordan elimination concerns the application of $G_n$ to  $L^{(n-1)}_i$ resulting in $L^{(n)}_i$, $0\le i\le n-1$.
As no further normalization is required (see Subsection \ref{sec:JordanElimination}), upon completion of Jordan elimination, the derived matrix is in LRRF denoted by
\begin{equation} \label{eq-fmat}
{\mathcal{L}}^{(n)}=\left(\begin{array}{l}
                                           L^{(n)}_{0} \vspace{0.03in}\\
                                           L^{(n)}_{1}\vspace{-0.05in}\\
                                          ... \\
                                           L^{(n)}_{n-1} \vspace{0.03in}\\
                                           G_{n}
\end{array}\right)
\end{equation}
At this point, the $n$-th stage of the process is completed and the subsequent stage starts by including the next nonzero row of $C$.
\subsection{Gaussian Elimination}
\label{sec:GaussianElimination}
The Gaussian Elimination process carries out the reduction of $C_{n}$ to $G_{n}$. Evidently, if $C_{n}=\textbf{0}$, then $G_n=\textbf{0}$. In all that follows we shall assume that $C_{n}\not=\textbf{0}$. As $L^{(n-1)}_i$ for $0\leq i\leq n-1$ remain unchanged at this stage of process, we employ shorthand notation for the row-length of $L^{(n-1)}_i$ by setting $\rho_i=\maxs (L^{(n-1)}_i)$ (instead of $\rho^{(n-1)}_i$) and $\rho_{n}=\maxs (G_{n})$. If $L^{(n-1)}_i\not=\textbf{0}$, then these are the pivot rows used by the Gaussian elimination.

Let $0\leq i\leq n-1$, $L^{(n-1)}_i\not=\textbf{0}$ and $L^{(i-1)}_{n}\not=\textbf{0}$. The reduction of $L^{(i-1)}_{n}$ to $L^{(i)}_{n}$ through the pivot $L^{(n-1)}_i$, can be described by the relation
\begin{equation}\label{eq-GJ01}
L^{(i)}_{n}=L^{(i-1)}_{n}+\lambda^{(i-1)}_{n}L^{(n-1)}_i,
\end{equation}
where $\lambda^{(i-1)}_{n}:=-l^{(i-1)}_{n\rho_i}$. In terms of scalar coordinates, (\ref{eq-GJ01}) takes the form
\begin{equation}\label{eq-GJ001}
l^{(i)}_{nk}=l^{(i-1)}_{nk}+\lambda^{(i-1)}_{n}l^{(n-1)}_{ik}
\end{equation}
for $k\in \omega$. Formally, if $\rho_i > \maxs(L^{(i-1)}_n)$, then $l^{(i-1)}_{n\rho_i}=0$, and so if $\lambda^{(i-1)}_{n}\not=0$, then  $\rho_i\le \maxs(L^{(i-1)}_n)$. In this latter case, as $l^{(n-1)}_{i\rho_i}$ is the rightmost one of $L^{(n-1)}_i$, the $(n,\rho_i)$ entry of $L^{(i-1)}_{n}$ is eliminated by the $(i, \rho_i)$ pivot element of $L^{(n-1)}_i$, according to $l^{(i)}_{n\rho_i}=l^{(i-1)}_{n\rho_i}+\lambda^{(i-1)}_{n}l^{(n-1)}_{i\rho_i}=l^{(i-1)}_{n\rho_i}-l^{(i-1)}_{n\rho_i}\cdot 1=0$. Hence for all $i$ such that $0\leq i\leq n-1$
\begin{equation}\label{eq-GJ2}
                                              l^{(i)}_{n\rho_i}=0.
\end{equation}                                      
As a consequence, $L^{(i)}_{n}\not=L^{(i-1)}_{n}$ if and only if $L^{(n-1)}_i\not=\textbf{0}$ and $\lambda^{(i-1)}_{n}\not=0$, and we shall refer to this transition as an \emph{effective reduction} of $L^{(i-1)}_n$ to $L^{(i)}_n$ through the row pivot $L^{(n-1)}_i$. 
\newcounter{nroman2}
The following cases exhaust all the possibilities:
\begin{list} {\roman{nroman2})}{\usecounter{nroman2}\setlength{\rightmargin}{\leftmargin}}
\item If $\maxs(L^{(i-1)}_n)< \rho_i$, then $\lambda^{(i-1)}_n=0$. Hence (\ref{eq-GJ01}) gives $L^{(i)}_n=L^{(i-1)}_n$ and no effective reduction occurs. Certainly $\maxs(L^{(i-1)}_n)=\maxs(L^{(i)}_n)$.
\item If $\maxs(L^{(i-1)}_n)=\rho_i$, then $\lambda^{(i-1)}_n\not=0$. Hence (\ref{eq-GJ01}) gives $L^{(i)}_n\not=L^{(i-1)}_n$ and an effective reduction occurs. Certainly $\maxs(L^{(i-1)}_n)>\maxs(L^{(i)}_n)$.
\item If $\maxs(L^{(i-1)}_n)> \rho_i$, then an effective reduction may ($\lambda^{(i-1)}_n\not=0$) or may not ($\lambda^{(i-1)}_n=0$) occur. In both cases: $\maxs(L^{(i-1)}_n)=\maxs(L^{(i)}_n)$.
\end{list}
The above results entail the inequality:
  \begin{equation} \label{Eq-ineq}
  \begin{split} 
 \maxs(G_n)\le ...&\le\maxs(L^{(i)}_n)\le \maxs(L^{(i-1)}_n)\le ...\le\maxs(C_n). 
 \end{split}
\end{equation}

Let $0\leq j\leq n-1$ and $L^{(n-1)}_{j}\not=\textbf{0}$. Since \vspace{0.03in}${\mathcal{L}}^{(n-1)}$ is in LRRF and $l^{(n-1)}_{j\rho_{j}}$ is a rightmost one, it follows that all the entries of ${\mathcal{L}}^{(n-1)}$ positioned at $(i, \rho_{j})$ with $i \not= j$ and $0\leq i\leq n-1$ are zero, that is
\begin{equation} \label{eq-G3}
                                             l^{(n-1)}_{i \rho_{j}}=0.
\end{equation}
Setting $k=\rho_j$ in (\ref{eq-GJ001}), on account of (\ref{eq-G3}), for every $i$ such that $0\leq i < j\le n-1$, it follows that $l^{(i)}_{n\rho_{j}}=l^{(i-1)}_{n\rho_{j}}$. The latter applied with $i=0,...,j-1$ gives: $c_{n\rho_{j}}=l^{(-1)}_{n\rho_{j}}=l^{(0)}_{n\rho_{j}}=...= l^{(j-1)}_{n\rho_{j}}$.
As an immediate consequence, the factor $\lambda^{(i-1)}_n$ in (\ref{eq-GJ01}) and (\ref{eq-GJ001}) is given by
\begin{equation} \label{eq-G40}
\lambda^{(i-1)}_n=-c_{n\rho_i}.
\end{equation}

Let us now consider the case in which the index $i$ is the immediate successor of $j$ such that $L^{(i)}_n\not=L^{(j)}_n$. Typically (\ref{eq-GJ01}) takes the form 
\begin{equation}\label{eq-GJ1}
L^{(i)}_n=L^{(j)}_n+\lambda^{(j)}_nL^{(n-1)}_i,
\end{equation}
where \vspace{0.03in}$0\leq j < i\le n-1$. On account of (\ref{eq-GJ2}) and (\ref{eq-G3}), (\ref{eq-GJ1}) implies that $l^{(i)}_{n\rho_{j}}=l^{(j)}_{n\rho_{j}}+\lambda^{(j)}_nl^{(n-1)}_{i \rho_{j}}=0$\vspace{0.02in}. This entails that the Gaussian elimination does not change zero entries previously obtained by corresponding pivot rows. Accordingly
\begin{equation} \label{eq-G4}
                                             l^{(i)}_{n \rho_{j}}=0
\end{equation}
for any $i,j$ such that $0\leq j < i\leq n-1$. Setting $i=n-1$ in (\ref{eq-G4}) we infer $l^{(n-1)}_{n\rho_{j}}=0$ for every $j$ such that $0\leq j < n-1$ and so $g_{n\rho_{j}}=0$.
Moreover, if $L^{(n-1)}_{n-1}\not=\textbf{0}$, then (\ref{eq-GJ2}) applied with $i=n-1$ gives $l^{(n-1)}_{n\rho_{n-1}}=0$ and so $g_{n\rho_{n-1}}=0$. From the previous two statements, we conclude that for all $j$ such that $0 \leq j \leq n-1$
\begin{equation} \label{Eq-JG7}
                                                g_{n\rho_{j}}=0.   
\end{equation}
Trivially (\ref{eq-GJ01}) implies: $\displaystyle L^{(n-1)}_{n}=L^{(-1)}_n+\sum^{n-1}_{i=0}\lambda^{(i-1)}_nL^{(n-1)}_i$.
On account of (\ref{eq-G40}), the latter relation takes the form: 
\begin{equation}\label{Eq-GJ2}
                                  L^{(n-1)}_{n}=C_n-\sum^{n-1}_{i=0}c_{n\rho_i}L^{(n-1)}_i.
\end{equation}
The relation (\ref{Eq-GJ2}) indicates the elementary operations involved in the row reduction of $C_n$ to $L^{(n-1)}_{n}$. These operations, followed by a scaling operation, result in $G_n$, as indicated in (\ref{leading}). Let $({\mathcal{L}}^{(n-1)}:G_n)$ be the augmented matrix derived by appending the row $G_n$ below the matrix ${\mathcal{L}}^{(n-1)}$. This matrix is not in general in LRRF, since at this stage of progress the Jordan elimination has not started to operate yet. However $({\mathcal{L}}^{(n-1)}:G_n)$ shares a common property with ${\mathcal{L}}^{(n)}$, as shown below.
\begin{proposition} \label{Gpropfth} Let $J_n$ be the indexing set of nonzero rows of the matrix $({\mathcal{L}}^{(n-1)}:G_n)$. If $(\rho_{j})_{j\in J_n}$ is the sequence of row-length of $({\mathcal{L}}^{(n-1)}:G_n)$, then the map $\rho:J_n\ni i \mapsto \rho_i \in \omega$ is injective.
\end{proposition}
\begin{proof}
Case \ 1. Let $i,j \in J_n$ such that $i\not= j$ with $i<n$ and $j<n$. As ${\mathcal{L}}^{(n-1)}$ is in LRRF, it follows that $\rho_i\not=\rho_{j}$, whence the assertion. Case \ 2. Let $i=n$ and $j<i$. If $G_n=\textbf{0}$, then $n\not\in J_n$, whence the assertion follows. If $G_n\not=\textbf{0}$, then $n\in J_n$. Certainly $g_{n\rho_n}\not=0$, as being the right leading one of $G_n$. On the contrary we assume that $\rho_{j}=\rho_n$. As $0\le j\le n-1$, on account of (\ref{Eq-JG7}), the contradiction follows from $g_{n\rho_n}=g_{n\rho_{j}}=0$. 
\end{proof}
\subsection{Jordan Elimination}
\label{sec:JordanElimination}
The Jordan elimination comprises the second part of the $n$-th stage of the process. If $G_n\not=\textbf{0}$, $G_n$ is used as a pivot row to eliminate the entries of preceding rows. Taking into account that the preceding rows of $G_n$ are of the form $L^{(n-1)}_i$ for $0\leq i\leq n-1$, the reduction of $L^{(n-1)}_i$ to $L^{(n)}_i$ through  $G_n$, can be described by
\begin{equation} \label{eq-gj2}  
L^{(n)}_i=L^{(n-1)}_i+\lambda^{(n-1)}_iG_n,      
\end{equation}
where \vspace{0.03in}$\lambda^{(n-1)}_i=-l^{(n-1)}_{i\rho_n}$. By analogy to what was previously said, the effective reduction of the $i$-th row through $G_n$, that is $L^{(n)}_i\not=L^{(n-1)}_i$, occurs if and only if $G_n\not=\textbf{0}$ and $\lambda^{(n-1)}_i\not=0$. 
If $L^{(n-1)}_i\not=\textbf{0}$ for $0\le i \le n-1$, Proposition \ref{Gpropfth} implies that $\maxs (G_n)=\rho_n\not=\rho_i=\maxs (L^{(n-1)}_i)$.
Besides this, if $\maxs (G_n) > \maxs (L^{(n-1)}_i)$, then $\lambda^{(n-1)}_i=l^{(n-1)}_{i\rho_n}=0$ and hence (\ref{eq-gj2}) implies that $L^{(n)}_i=L^{(n-1)}_i$. From the last two statements we deduce that, if $L^{(n)}_i\not=L^{(n-1)}_i$, then
\begin{equation} \label{eq-gj42}
\maxs (G_n) < \maxs (L^{(n-1)}_i)
\end{equation}
for $0\le i \le n-1$. We infer from (\ref{eq-gj2}) and (\ref{eq-gj42}) that in all cases ($L^{(n)}_i\not=L^{(n-1)}_i$ or $L^{(n)}_i=L^{(n-1)}_i$)
\begin{equation} \label{eq-gj5}
 \maxs (L^{(n-1)}_i)= \maxs (L^{(n)}_i)
 \end{equation}
for all $n=1,2,...$ and all $i: 0\le i \le n-1$. Thus $\maxs (L^{(i)}_i)= \maxs (L^{(i+1)}_i)=...$ and so 
\begin{equation} \label{Eq-gj51}
\maxs (G_i)= \maxs (L^{(n)}_i)  \   \  \ \forall \ n>i.
\end{equation}
In view of Example \ref{example1}, $U_n$ in (\ref{eq-ex1.matrU}) shows that the analog of (\ref{Eq-gj51}) does not hold, if the algorithm is implemented with LPS. We conclude that only the Gaussian elimination can generate new zero rows and new row lengths. The latter along with the fact that the Gaussian elimination acts only once on a nonzero row entails that $\rho_i=\maxs (L^{(n)}_i)$ for all $n$ with $n\ge i$ and so the shorthand notation for $\rho_i=\maxs(G_i)$ employed in Subsection \ref{sec:GaussianElimination} can also be used for all $n>i$. Furthermore, since the Jordan elimination does not affect right leading ones of preceding rows, in each stage the normalization process needs just one scaling operation given by (\ref{leading}). By (\ref{eq-gj42}), (\ref{eq-gj5}) and (\ref{Eq-gj51}), respectively, we deduce that if $L^{(n)}_i\not=L^{(n-1)}_i$, then $\maxs (G_n) < \maxs (L^{(n-1)}_i)=\maxs (L^{(n)}_i)=\maxs (G_i)$ for $i\le n-1$. Thus
\begin{equation} \label{Eq-gj52}
\maxs (G_n)< \maxs (G_i) \ \ \ (i\le n-1),
\end{equation}
whenever $G_n$ had effectively eliminated an element of the $i$-th row. 
\begin{corollary} \label{proposition zero-rows}
For every  $n>i$
\begin{equation} \label{Eq-gj53}
                                 G_i=\mbox{\boldmath $\mathrm{0}$}\Longleftrightarrow L^{(n)}_i= \mbox{\boldmath $\mathrm{0}$}.
\end{equation}
\end{corollary}
\begin{proof} Since $n>i$, the direct implication follows from the fact that algorithm leaves zero rows unchanged. The converse follows directly from (\ref{Eq-gj51}).
\end{proof}
By virtue of (\ref{Eq-ineq}), (\ref{Eq-gj51}) implies $\maxs (L^{(n)}_i)\leq \maxs (C_i)$ for all $i\in \omega$ and $n\in \omega$. 
\begin{proposition} \label{proposition G_i}
If $G_i\not=\mbox{\boldmath $\mathrm{0}$}$ and $G_{j}\not=\mbox{\boldmath $\mathrm{0}$}$ for any $i,j$ in $\omega$ with $i\not= j$, then 
\begin{equation} \label{Eq-gjG_i}
\maxs (G_i)\not= \maxs (G_{j}).
\end{equation}
Equivalently the map $J\ni i\mapsto \maxs (G_i)\in \omega$ is injective.
\end{proposition}
\begin{proof} Let $n>\max\{i,j\}$. Corollary \ref{proposition zero-rows} implies: $L^{(n)}_i\not=\textbf{0}$ and $L^{(n)}_{j}\not=\textbf{0}$. As $L^{(n)}_i, L^{(n)}_{j}$ are rows of the matrix ${\mathcal{L}}^{(n)}$, which is in LRRF, we conclude that $\maxs (L^{(n)}_i)\not= \maxs (L^{(n)}_{j})$. Hence, the assertion follows from (\ref{Eq-gj51}).
\end{proof}
\subsection{The Main Theorem and Some Consequences}
\label{sec:TheMainTheoremAndSomeConsequences}
Theorem \ref{main-theorem} below shows that an arbitrary row of the form $L^{(n)}_{k}$ remains invariant after a sufficient large $n\in \omega$. It demonstrates the main difference between LPS and RPS in the implementation of the infinite Gauss-Jordan algorithm.  
\begin{theorem} \label{main-theorem}
For every $k\in \omega$ there exists ${\mbox{\scriptsize $\mathrm{M}$}}_{k}\in \omega$ with ${\mbox{\scriptsize $\mathrm{M}$}}_{k}\ge k$ such that
\begin{equation} \label{eq-main} \forall n: n > {\mbox{\scriptsize $\mathrm{M}$}}_{k} \Rightarrow  L^{(n)}_{k}=L^{({\mbox{\tiny $\mathrm{M}$}}_{k})}_{k}
\end{equation} 
\end{theorem}
\begin{proof} 
On the contrary we assume the existence of some $k\in \omega$ such that for every ${\mbox{\scriptsize $\mathrm{M}$}}\ge k $
\begin{equation} \label{eq-contradiction}
  \exists n_{\mbox{\tiny $\mathrm{M}$}}: n_{\mbox{\tiny $\mathrm{M}$}}> \mbox{\scriptsize $\mathrm{M}$} \Rightarrow  L^{(n_{\mbox{\tiny $\mathrm{M}$}})}_{k} \not=L^{(\mbox{\tiny $\mathrm{M}$})}_{k}. 
\end{equation} 
With the aid of (\ref{eq-contradiction}) a sequence of natural numbers is constructed as follows: Applying (\ref{eq-contradiction}) with ${\mbox{\scriptsize $\mathrm{M}$}}=k+1$, then 
\[ \exists n_{k+1}: n_{k+1}>k+1 \Rightarrow  L^{(n_{k+1})}_{k} \not=L^{(k+1)}_{k}.     \]
Let ${\mbox{\scriptsize $\mathrm{M}$}}_{1}={\rm min}\{n_{k+1}>k+1 : L^{(n_{k+1})}_{k} \not=L^{(k+1)}_{k}\}$. Then applying (\ref{eq-contradiction}) with ${\mbox{\scriptsize $\mathrm{M}$}}= {\mbox{\scriptsize $\mathrm{M}$}}_{1}$ we have
\[ \exists n_{{\mbox{\tiny $\mathrm{M}$}}_{1}}:  n_{{\mbox{\tiny $\mathrm{M}$}}_{1}}> {\mbox{\scriptsize $\mathrm{M}$}}_{1} \Rightarrow \ L^{(n_{{\mbox{\tiny $\mathrm{M}$}}_{1}})}_{k}\not=L^{({\mbox{\tiny $\mathrm{M}$}}_{1})}_{k}  \]
and we define ${\mbox{\scriptsize $\mathrm{M}$}}_{2}={\rm min}\{n_{{\mbox{\tiny $\mathrm{M}$}}_{1}}> {\mbox{\scriptsize $\mathrm{M}$}}_{1}: L^{(n_{{\mbox{\tiny $\mathrm{M}$}}_{1}})}_{k}\not=L^{({\mbox{\tiny $\mathrm{M}$}}_{1})}_{k} \}$ and the process continues ad infinitum. Let ${\mbox{\scriptsize $\mathrm{M}$}}_0=k+1$. Evidently ${\mbox{\scriptsize $\mathrm{M}$}}_0<{\mbox{\scriptsize $\mathrm{M}$}}_1<{\mbox{\scriptsize $\mathrm{M}$}}_2,...$. For each ${\mbox{\scriptsize $\mathrm{M}$}}_i={\rm min}\{n_{{\mbox{\tiny $\mathrm{M}$}}_{i-1}}>{\mbox{\tiny $\mathrm{M}$}}_{i-1}: L^{(n_{{\mbox{\tiny $\mathrm{M}$}}_{i-1}})}_{k}\not=L^{({\mbox{\tiny $\mathrm{M}$}}_{i-1})}_{k} \}$ the effective reduction of $L^{({\mbox{\tiny $\mathrm{M}$}}_{i-1})}_{k}$ to $L^{({\mbox{\tiny $\mathrm{M}$}}_i)}_{k}$ through the pivot $G_{{\mbox{\tiny $\mathrm{M}$}}_i}$ can be described by 
\[L^{({\mbox{\tiny $\mathrm{M}$}}_i)}_{k}=L^{({\mbox{\tiny $\mathrm{M}$}}_{i-1})}_{k}+\lambda^{({\mbox{\tiny $\mathrm{M}$}}_{i-1})}_{k} G_{{\mbox{\tiny $\mathrm{M}$}}_i}\]
Certainly $L^{({\mbox{\tiny $\mathrm{M}$}}_i)}_{k}\not=L^{({\mbox{\tiny $\mathrm{M}$}}_{i-1})}_{k}$ and ${\mbox{\scriptsize $\mathrm{M}$}}_{i-1}<{\mbox{\scriptsize $\mathrm{M}$}}_i$. It turns out that there exists an infinite sequence $G_{{\mbox{\tiny $\mathrm{M}$}}_{1}}, G_{{\mbox{\tiny $\mathrm{M}$}}_{2}}, ...$ of successive pivot rows, which effectively eliminate entries of the $k$-th row. Let us now consider the corresponding sequence of natural numbers $\maxs (G_{{\mbox{\tiny $\mathrm{M}$}}_{1}}),\maxs (G_{{\mbox{\tiny $\mathrm{M}$}}_{2}}), ...$. Proposition \ref{proposition G_i} implies that $\maxs (G_{{\mbox{\tiny $\mathrm{M}$}}_i})\not=\maxs (G_{{\mbox{\tiny $\mathrm{M}$}}_{j}})$ for any $i,j$ in $\omega$ with $i\not= j$. Thus the sequence $(\maxs (G_{{\mbox{\tiny $\mathrm{M}$}}_i}))_{i\in \omega}$ determines an infinite subset of $\omega$, say ${\mathcal{M}}=\{\maxs (G_{{\mbox{\tiny $\mathrm{M}$}}_{1}}),\maxs (G_{{\mbox{\tiny $\mathrm{M}$}}_{2}}), ...\}$. Clearly the ordinal number of $\mathcal{M}$ is $\omega$, and we shall write for this ${\rm ord}({\mathcal{M}})=\omega$. We deduce from (\ref{Eq-gj52}) that 
\begin{equation} \label{Eq-gj520}
\maxs (G_{{\mbox{\tiny $\mathrm{M}$}}_i})<\maxs (G_{k})
\end{equation}
for all $i\in \omega$. Let us call $I(\rho_k)$ the initial segment of $\omega$ determined by the natural number $\rho_k=\maxs (G_{k})$. The definition of ordinals implies ${\rm ord}(I(\rho_{k}))=\rho_{k}$. We infer from (\ref{Eq-gj520}) that $\maxs (G_{{\mbox{\tiny $\mathrm{M}$}}_i})\in I(\rho_{k})$ for all $i\in \omega$, whence ${\mathcal{M}}\subset I(\rho_{k})$. However, the latter is a contradictory statement, inasmuch as 
$\omega={\rm ord}({\mathcal{M}})\leqq {\rm ord}(I(\rho_{k}))=\rho_{k}\lvertneqq\omega$.
\end{proof}
By virtue of Theorem \ref{main-theorem}, for any row of index $k\in \omega$, there exists $\mbox{\scriptsize $\mathrm{M}$}_k$ with $\mbox{\scriptsize $\mathrm{M}$}_k\ge k$ such that the row $L^{(n)}_{k}$ remains fixed for all $n\ge \mbox{\scriptsize $\mathrm{M}$}_k$. We thus can define $L_{k}=L^{({\mbox{\tiny $\mathrm{M}$}}_k)}_k$ for the above $\mbox{\scriptsize $\mathrm{M}$}_{k}\in \omega$. Accordingly the infinite Gauss-Jordan algorithm implemented with RPS results in a row-finite $\omega\times\omega$ matrix $L:=(L_k)_{k\in \omega}$. 
Certainly each row $L_k$ of $L$ is constructed after a finite number of algorithmic steps. Thus $L_k$ is ultimately a finite linear combination of $C_i$. If $\mbox{\scriptsize $\mathrm{M}$}_k$ is defined to be the smallest integer such that ${\mbox{\scriptsize $\mathrm{M}$}}_{k}\ge k$ and $L^{(n)}_{k}=L^{({\mbox{\tiny $\mathrm{M}$}}_{k})}_{k}$ for all $n\ge {\mbox{\scriptsize $\mathrm{M}$}}_{k}$, then
\begin{equation} \label{main2}
L_{k}=L^{({\mbox{\tiny $\mathrm{M}$}}_k)}_k=\sum^{\mbox{\tiny $\mathrm{M}$}_k}_{i=0} \lambda_{ki}C_{i}.
\end{equation}
The subsequent corollary follows immediately:
\begin{corollary} \label{cor-GJ1}
For every $k\in \omega$, $L_{k}$ is a finite linear combination of $C_i$. Equivalently $\RS(L)\subseteq \RS(C)$.
\end{corollary}
\begin{corollary} \label{cor-GJ3} The following statements hold:
\emph{i)} $\Right(G)=\Right(L)=\Right(\RS(G))=\Right(\RS(L))$. \emph{ii)} $G_i=\mbox{\boldmath $\mathrm{0}$}$ if and only if $L_i=\mbox{\boldmath $\mathrm{0}$}$.
\end{corollary}
\begin{proof}
i) For an arbitrary $i\in \omega$, it follows from (\ref{Eq-gj51}) that $\maxs (G_i)=\maxs (L^{(n)}_i)$ for all $n\ge i$. Applying Theorem \ref{main-theorem} with  $n\ge {\mbox{\scriptsize $\mathrm{M}$}}_i\ge i$, we conclude that $\maxs (L^{(n)}_i)=\maxs (L_i)$.  Thus $\maxs (G_i)=\maxs (L_i)$ for all $i\in \omega$ and the first equality follows. As the maps of row-length of the nonzero rows of $G$ and $L$ coincide, it follows from Proposition \ref{proposition G_i} that they are injective.
Hence the last two equalities follow from Corollary \ref{corollary_1.4}.
ii) It follows from (\ref{Eq-gj53}).
\end{proof}
\begin{theorem} 
The row-finite matrix $L$ is in \emph{LRRF}.
\end{theorem}
\begin{proof}
Let $L_n$ be an arbitrary nonzero row of $L$ and $\rho_n=\maxs (L_n)$. As $l_{n\rho_n}$ stands for the rightmost one of $L_n$, we are to show that $l_{k\rho_n}=0$ for all $k\in \omega$ such that $k\not=n$. In view of Theorem \ref{main-theorem}, if $m>\max\{{\mbox{\scriptsize $\mathrm{M}$}}_{k}, {\mbox{\scriptsize $\mathrm{M}$}}_n\}$, then $L_{k}=L^{(m)}_{k}$ and $L_n=L^{(m)}_n$. Thus $L_k,\; L_n$ are rows of ${\mathcal{L}}^{(m)}$ and $l_{k\rho_n}=l^{(m)}_{k\rho_n}$. Since ${\mathcal{L}}^{(m)}$ is in LRRF, it follows that $l^{(m)}_{k\rho_n}=0$, as claimed.
\end{proof} 
\subsection{The Row Equivalence of $C$ and $L$}
\label{sec:TheRowEquivalence}
In the sequel, the row equivalence of $C$ and $L$ is established. Let $f$ be the induced linear mapping by $C$. Thus $f(\mbox{\boldmath $\mathrm{e}$}_{n})=C_n$ for all $n\in \omega$ and $C=[f]^{\mbox{\scriptsize \boldmath $\mathrm{E}$}}_{\mbox{\scriptsize \boldmath $\mathrm{E}$}}$. Let $\mbox{\boldmath $\mathrm{L}$}=(\mbox{\boldmath $\mathrm{l}$}_{k})_{k\in \omega}$ be the sequence derived by applying the same sequence of row operations to the corresponding terms of $\mbox{\boldmath $\mathrm{E}$}$, which have occurred in the reduction of $C$ to $L$. In view of (\ref{main2}) $\mbox{\boldmath $\mathrm{l}$}_k$ is given by $\mbox{\boldmath $\mathrm{l}$}_{k}=\sum^{\mbox{\tiny $\mathrm{M}$}_k}_{i=0} \lambda_{ki}\mbox{\boldmath $\mathrm{e}$}_{i}$. Since \vspace{0.03in}$f(\mbox{\boldmath $\mathrm{l}$}_{k})=\sum^{\mbox{\tiny $\mathrm{M}$}_k}_{i=0} \lambda_{ki}f(\mbox{\boldmath $\mathrm{e}$}_{i})=\sum^{\mbox{\tiny $\mathrm{M}$}_k}_{i=0} \lambda_{ki}C_{i}=L_k$, we conclude that each nonzero row $L_{k}$, $k\in J$, represents the element $f(\mbox{\boldmath $\mathrm{l}$}_{k})$, $k\in J$, relative to $\mbox{\boldmath $\mathrm{E}$}$. Let also $G=(G_{k})_{k\in \omega}$ be the matrix whose rows are derived by the Gaussian elimination, as defined by (\ref{leading}) and (\ref{Eq-GJ2}). Corollary \ref{cor-GJ3} entails that $G, L$ are associated with the same indexing sets of zero and nonzero rows designated by $W, J$, respectively and that $G, L$ are associated with identical sequences of row-length. In analogy to $\mbox{\boldmath $\mathrm{L}$}$, the sequence $\mbox{\boldmath $\mathrm{G}$}=(\mbox{\boldmath $\mathrm{g}$}_{k})_{k\in \omega}$, is defined by applying the same sequence of row operations to $\mbox{\boldmath $\mathrm{E}$}$, which have occurred in the reduction of $C$ to $G$. Certainly $f(\mbox{\boldmath $\mathrm{g}$}_{n})=G_{n}$, $n\in \omega$.
\begin{proposition} \label{prop4}
\emph{i)} The set $\mbox{\boldmath $\mathrm{G}$}$ is a basis of ${\mathcal{F}}^{(\omega)}$.
\emph{ii)} $G\sim C$.
\emph{iii)} $\RS(G)= \RS(C)$.
\emph{iv)} The set $\{G_{j}\}_{j\in J}$ of nonzero rows of $G$ is a complete basis of  row-length of $\RS(C)$.
\emph{v)} The set $\{\mbox{\boldmath $\mathrm{g}$}_{w}\}_{ w\in W}$ is a basis of $\Ker(f)$.
\end{proposition}
\begin{proof} i) If $C_n=\textbf{0}$, then $G_n=\textbf{0}$ and so $\mbox{\boldmath $\mathrm{g}$}_n=\mbox{\boldmath $\mathrm{e}$}_n$. Let $C_n\not=\textbf{0}$.
Since each term $L^{(n-1)}_i$ in (\ref{Eq-GJ2}) is ultimately a linear combination of $C_{0},...,C_{n-1}$, we can write $L^{(n-1)}_{n}=C_n-\sum^{n-1}_{k=0}\alpha_{nk}C_k$. 
By applying the same sequence of row operations to the corresponding terms of $\mbox{\boldmath $\mathrm{E}$}$, we reach $\mbox{\boldmath $\mathrm{v}$}_n=\mbox{\boldmath $\mathrm{e}$}_{n}-\sum^{n-1}_{k=0}\alpha_{nk}\mbox{\boldmath $\mathrm{e}$}_k$. In view of (\ref{leading}), if $L^{(n-1)}_n\not=\textbf{0}$, then $\mbox{\boldmath $\mathrm{g}$}_n=\frac{1}{l^{(n-1)}_{n\rho_n}}\mbox{\boldmath $\mathrm{v}$}_n$, otherwise $G_n=L^{(n-1)}_n=\textbf{0}$ and $\mbox{\boldmath $\mathrm{g}$}_n=\mbox{\boldmath $\mathrm{v}$}_n$. Thus writing $\mbox{\boldmath $\mathrm{g}$}_n=\sum^n_{k=0} \gamma_{nk}\mbox{\boldmath $\mathrm{e}$}_k$\vspace{0.03in}, if $G_n=\textbf{0}$ the leading coefficient of $\mbox{\boldmath $\mathrm{g}$}_n$ is $\gamma_{nn}=1$, otherwise $\gamma_{nn}=\frac{1}{l^{(n-1)}_{n\rho_n}}$. Accordingly $(\gamma_{nk})_{(n,k)\in \omega\times\omega}$ is a lower triangular matrix with nonzero elements in the diagonal. Thus $\mbox{\boldmath $\mathrm{G}$}=(\mbox{\boldmath $\mathrm{g}$}_{n})_{n\in \omega}$ satisfies (\ref{eq-Bou1.1}), and the assertion follows from Proposition \ref{proposition_1.2}.
ii) Let $Q$ be the matrix of passage from $\mbox{\boldmath $\mathrm{E}$}$ to $\mbox{\boldmath $\mathrm{G}$}$. According to the previous statement $G=[f]_{\mbox{\scriptsize \boldmath $\mathrm{G}$}}^{\mbox{\scriptsize \boldmath $\mathrm{E}$}}$. Thus (\ref{change of basis}) implies that $G=Q\cdot C$ and so $G\sim C$.
iii) It follows directly from Proposition \ref{proposition_3.1}.
iv) In view of Proposition \ref{proposition G_i}, $J\ni i\mapsto \maxs (G_i)\in \omega$ is injective. As $\RS(G)=\RS(C)$, the set $\{G_{j}\}_{j\in J}$ is a generating system of $\RS(C)$. Corollary \ref{cor-GJ3} and Proposition \ref{proposition_1.3} imply the assertion.
v) It suffices to show that the subset $\{\mbox{\boldmath $\mathrm{g}$}_{w}\}_{w\in W}$ of $\mbox{\boldmath $\mathrm{G}$}$ is a generating system of $\NS(G)$. Let $\mbox{\boldmath $\mathrm{n}$}\in \NS(G)$. Writing $\mbox{\boldmath $\mathrm{n}$}$ in terms of the basis $\mbox{\boldmath $\mathrm{G}$}$ we have:
\begin{equation} \label{eq-nullbasis} \mbox{\boldmath $\mathrm{n}$}=\sum_{i\in \omega}\lambda_i\mbox{\boldmath $\mathrm{g}$}_i=\sum_{i\in J}\lambda_i\mbox{\boldmath $\mathrm{g}$}_i+\sum_{i\in W}\lambda_i\mbox{\boldmath $\mathrm{g}$}_{i}.
\end{equation}
On account of $f(\mbox{\boldmath $\mathrm{g}$}_w)=\textbf{0}$ for all $w\in W$, we conclude that
\[f(\mbox{\boldmath $\mathrm{n}$})=f(\sum_{i\in J}\lambda_i\mbox{\boldmath $\mathrm{g}$}_i)+f(\sum_{i\in W}\lambda_i\mbox{\boldmath $\mathrm{g}$}_{i})=\sum_{i\in J}\lambda_i f(\mbox{\boldmath $\mathrm{g}$}_i)=\sum_{i\in J}\lambda_i G_i.\]
But also $f(\mbox{\boldmath $\mathrm{n}$})=\textbf{0}$ and so $\sum_{i\in J}\lambda_i G_i=0$. The linear independence of $(G_j)_{j\in J}$ entails that $\lambda_i=\textbf{0}$ for all $i\in J$ and so (\ref{eq-nullbasis}) implies $\mbox{\boldmath $\mathrm{n}$}=\sum_{i\in W}\lambda_i\mbox{\boldmath $\mathrm{g}$}_{i}$, as required.
\end{proof}
\begin{proposition} \label{prop51} $\RS(L)=\RS(C)$.
\end{proposition}
\begin{proof} 
Corollary \ref{cor-GJ1} and Proposition \ref{prop4}(ii) imply, respectively, that $\RS(L)\subset \RS(C)=\RS(G)$. It remains to be shown that $\RS(G)\subset \RS(L)$. Corollary \ref{cor-GJ3} entails that $\Right(G)=\Right(L)=\Right(\RS(G))$. For simplifying notation, let us write: $\Right:=\Right(\RS(G))=\{\rho_j\}_{j\in J}$.
Let $g\in \RS(G)$ and
$g=\sum_{i=0}^{\rho_{k}} \alpha_i\mbox{\boldmath $\mathrm{e}$}_i$ 
for $\rho_k\in \Right$. Since $\rho:J\mapsto \Right$ is bijective, we can relabel the set of nonzero rows of $L$ by setting $L'_{\rho_j}:=L_{j}$  for all $j\in J$. Certainly $\maxs (L'_{\rho_j})=\rho_j$ for all $j\in J$ or $\maxs (L'_i)=i$ for all $i\in \Right$. Without loss of generality, we assume that $\rho_0<\rho_1<...<\rho_k$ and define $g'=\alpha_{\rho_0}L'_{\rho_0}+\alpha_{\rho_1}L'_{\rho_1}+...+\alpha_{\rho_k}L'_{\rho_k}$, or alternatively 
\vspace{0.05in}$\displaystyle g'=\sum^{\rho_k}_{\substack{i=0  \\  i\in \Right}} \alpha_iL'_i.$ 
Evidently $g'\in \RS(L)$ and so $g'\in \RS(G)$. We are to show that $g=g'$. Whereas $L$ is in LRRF, in each nonzero row $L'_i=(l'_{im})_{m\in\omega}$, $i\in\Right$, we have\vspace{0.05in}: $l'_{im}=0$  for all $m\in \Right$ with $m\not=i$ and  $l'_{ii}=1$.  Thus, $L'_i$ can be written as 
\vspace{0.05in}$\displaystyle L'_i=\mbox{\boldmath $\mathrm{e}$}_i+\sum^{i-1}_{\substack{ m=0 \\ m\not\in \Right}}l'_{im}\mbox{\boldmath $\mathrm{e}$}_{m}$
for all $i\in \Right$. On account of
\begin{equation*} \begin{split}
\displaystyle \sum^{\rho_k}_{\substack{i=0  \\  i\in \Right}} \alpha_iL'_i=&\sum^{\rho_k}_{\substack{i=0  \\  i\in \Right}} \alpha_i(\mbox{\boldmath $\mathrm{e}$}_i+\sum^{i-1}_{\substack{ m=0 \\ m\not\in \Right}}l'_{im}\mbox{\boldmath $\mathrm{e}$}_{m})=\sum^{\rho_k}_{\substack{i=0  \\  i\in \Right}} (\alpha_i\mbox{\boldmath $\mathrm{e}$}_i+\sum^{i-1}_{\substack{ m=0 \\ m\not\in \Right}}\alpha_il'_{im}\mbox{\boldmath $\mathrm{e}$}_{m}) \vspace{0.05in}\\
=& \sum^{\rho_k}_{\substack{ i=0 \\ i\in \Right}}\alpha_i\mbox{\boldmath $\mathrm{e}$}_i+\sum^{\rho_k}_{\substack{i=0 \\ i\in \Right}}\sum^{i-1}_{\substack{m=0 \\ m\not\in \Right}} \alpha_i l'_{i m}\mbox{\boldmath $\mathrm{e}$}_m
\end{split}
\end{equation*}
we deduce that
\begin{equation*} 
\begin{split} g-g' =& \sum_{i=0}^{\rho_k} \alpha_i\mbox{\boldmath $\mathrm{e}$}_i-\sum^{\rho_k}_{\substack{i=0  \\  i\in \Right}} \alpha_iL'_i \\
=& \sum^{\rho_k}_{\substack{i=0 \\ i\in \Right}}\alpha_i\mbox{\boldmath $\mathrm{e}$}_i+\sum^{\rho_k}_{\substack{m=0 \\ m\not\in \Right}} \alpha_m\mbox{\boldmath $\mathrm{e}$}_m-\left(\sum^{\rho_k}_{\substack{ i=0 \\ i\in \Right}}\alpha_i\mbox{\boldmath $\mathrm{e}$}_i+\sum^{\rho_k}_{\substack{i=0 \\ i\in \Right}}\sum^{i-1}_{\substack{m=0 \\ m\not\in \Right}} \alpha_i l'_{i m}\mbox{\boldmath $\mathrm{e}$}_m \right)\\
=& \sum^{\rho_k}_{ \substack{ m=0 \\  m\not\in \Right}} \alpha_m\mbox{\boldmath $\mathrm{e}$}_m-\sum^{\rho_k}_{\substack{i=0 \\ i\in \Right}}\sum^{i-1}_{\substack{m=0 \\ m\not\in \Right}} \alpha_i l'_{i m}\mbox{\boldmath $\mathrm{e}$}_m 
\end{split}
\end{equation*}
It follows immediately that $g-g'$ is a linear combination of $\mbox{\boldmath $\mathrm{e}$}_{m}$ with $m\not\in \Right$, which in turn implies that $\maxs (g-g')\not\in \Right$. However, since $g-g'\in \RS(G)$, it follows that $g-g'=0$; for if otherwise $\maxs (g-g')\in \Right$, that contradicts the previous result. Thus $g=g'$, and so $g\in \RS(L)$, as claimed.
\end{proof} 
\begin{proposition} \label{prop6} \emph{i)} The set of nonzero rows of $L$ is a complete basis of  row-length of $\RS(C)$. \emph{ii)} $L$ is row equivalent to $C$. \emph{ii)} $\mbox{\boldmath $\mathrm{L}$}$ is a basis of  ${\mathcal{F}}^{(\omega)}$.
\end{proposition}
\begin{proof} i) Whereas the map of row-length of the nonzero rows of $L$ is injective, Proposition \ref{proposition LREF_1} entails that $(L_{j})_{j\in J}$ is a complete basis of row-length of $\RS(L)$. Thus Proposition \ref{prop51} implies that $(L_{j})_{j\in J}$ is a complete basis of  row-length of $\RS(C)$, as asserted. ii) Corollary \ref{cor-GJ3} implies $\mbox{\boldmath $\mathrm{l}$}_{w}=\mbox{\boldmath $\mathrm{g}$}_{w}$ for all $w\in W$. As a 
direct consequence $\nul(L)=\nul(G)={\rm card}(W)$. As $C\sim G$, we conclude that $\nul(L)=\nul(G)=\nul(C)$. Since $\RS(L)= \RS(C)$, Proposition \ref{proposition_3.1} shows the assertion. iii) Evidently $\mbox{\boldmath $\mathrm{L}$}=\{\mbox{\boldmath $\mathrm{l}$}_{w}\}_{w\in W}\cup \{ \mbox{\boldmath $\mathrm{l}$}_{j}\}_{j\in J}$. Proposition \ref{prop4} entails that $\{\mbox{\boldmath $\mathrm{g}$}_{w}\}_{w\in W}$ is a basis of $\NS(C)$ and so is $\{\mbox{\boldmath $\mathrm{l}$}_{w}\}_{w\in W}$. Moreover, $\{L_j\}_{j\in J}$ is a basis of $\RS(C)$ and $f(\mbox{\boldmath $\mathrm{l}$}_{j})=L_j, \ j\in J$. Proposition \ref{proposition_1.1} (i) entails that $\{\mbox{\boldmath $\mathrm{l}$}_{j}\}_{j\in J}$ is a basis of a complementary space of $\NS(C)$ and the assertion follows.
\end{proof}
As a general conclusion, a row equivalent LRRF, say $L$, of an arbitrary $C\in \RFM$ is constructed, by means of the infinite Gauss-Jordan algorithm implemented with RPS.
Moreover, Corollary \ref{corollary LRRF_4} entails the existence of a permutation matrix $P$ so that the matrix $H=P\cdot L$ is a QHF of $L$. Since $H\sim L\sim C$, it turns out that $H$ is a QHF of $C$. The latter provides a proof for the existence of QHFs, without invoking the axiom of countable choice.
\section{Construction of Submatrices of LRRF-QHF}
\label{sec:ConstructionOfQuasiHermiteForms}
From a computational point of view the main objective is the construction of a sequence of submatrices of a QHF, $H$, of $C\in \RFM$, as established in this Section. 
Following the notation of Section \ref{sec:TheInfiniteGJEliminationAndTheRowEquivalence}, we start with the construction of a chain of submatrices of a LRRF, $L$, of $C$. This is not a direct consequence of the algorithmic process formulated in the previous Section, since $L^{(n)}_i$ may differ from $L_i$ for small enough $n$ and so ${\mathcal{L}}^{(n)}$ may not be a submatrix of $L$. 
\subsection{Construction of Submatrices of LRRF}
\label{sec:ConstructionOfASequenceOfSubmatricesOfL}
Let $m\ge n$. In view of (\ref{eq-fmat}), it is convenient to introduce the matrix notation:
\begin{equation} \label{eq-fmat3}
{\mathcal{L}}^{(m)}\!\!\mid_n=\left(\begin{array}{l}
                                           L^{(m)}_{0} \\
                                           L^{(m)}_{1}\vspace{-0.05in}\\
                                          ... \\
                                           L^{(m)}_{n-1} \\
                                           L^{(m)}_n
\end{array}\right)
\end{equation}
Moreover $A\sqsubset B$ will indicate that $A$ is a strict submatrix of $B$. Certainly ${\mathcal{L}}^{(n)}\!\!\mid_n={\mathcal{L}}^{(n)}$. As ${\mathcal{L}}^{(m)}\!\!\mid_n\sqsubset{\mathcal{L}}^{(m)}\!\!\mid_m$ for $n\le m$, ${\mathcal{L}}^{(m)}\!\!\mid_n$ is in LRRF too. 
\begin{proposition} \label{prop-GJ2}
For every $n\in \omega$ there exists $\epsilon_n\in \omega$ with $\epsilon_n\ge n$ such that ${\mathcal{L}}^{(\epsilon_n)}\!\!\mid_n\sqsubset L$.
\end{proposition}
\begin{proof}
Let $n\in \omega$ and $0\leq k \leq n$. According to Theorem \ref{main-theorem}, for each $k$ there is ${\mbox{\scriptsize $\mathrm{M}$}}_k\in \omega$ such that $L^{({\mbox{\scriptsize $\mathrm{M}$}}_k)}_k=L_k$. Let $\epsilon_n=\max({\mbox{\scriptsize $\mathrm{M}$}}_k)_{0\leq k \leq n}$. As ${\mbox{\scriptsize $\mathrm{M}$}}_k \le \epsilon_n$, it follows that $L^{(\epsilon_n)}_{k}=L^{({\mbox{\scriptsize $\mathrm{M}$}}_k)}_{k}=L_{k}$. Thus ${\mathcal{L}}^{(\epsilon_n)}\!\!\mid_n=(L_{k})_{0\leq k \leq n}\sqsubset L$, as claimed.
\end{proof}
As a consequence of Proposition \ref{prop-GJ2}, a chain $\{{\mathcal{L}}^{(\epsilon_n)}\!\!\mid_n\}_{n\in \omega}$, is generated satisfying:
\begin{equation} \label{chain} 
{\mathcal{L}}^{(\epsilon_0)}\!\!\mid_0\sqsubset {\mathcal{L}}^{(\epsilon_1)}\!\!\mid_1\sqsubset ... \sqsubset {\mathcal{L}}^{(\epsilon_n)}\!\!\mid_n \sqsubset ... \sqsubset L
\end{equation}
Using set-theoretic notation ${\mathcal{L}}^{(\epsilon_n)}\!\!\mid_n=\{(k,L_k): 0\le k\le n\}=L\!\!\mid_n$, whence $\displaystyle L=\bigcup_{n\in \omega}{\mathcal{L}}^{(\epsilon_n)}\!\!\mid_n \in \RFM$.
The infinite Gauss-Jordan algorithm directly generates the sequence $\{{\mathcal{L}}^{(n)}\}_{n\in \omega}$ along with the subsequence $\{{\mathcal{L}}^{(\epsilon_n)}\}_{n\in \omega}$, but in general ${\mathcal{L}}^{(\epsilon_n)}\not\sqsubset L$. However, by virtue of Proposition \ref{prop-GJ2}, since ${\mathcal{L}}^{(\epsilon_n)}\!\!\mid_n\sqsubset {\mathcal{L}}^{(\epsilon_n)}$, this algorithm also generates the chain $\{{\mathcal{L}}^{(\epsilon_n)}\!\!\mid_n\}_{n\in \omega}$ of submatrices of $L$ satisfying (\ref{chain}). Accordingly, given any desired number of rows, say $n$,
if $N$ is large enough ($N\ge \epsilon_n$), the submatrix ${\mathcal{L}}^{(N)}|_n$ of ${\mathcal{L}}^{(N)}$ is also a submatrix of $L$.
\subsection{An Extension of the Infinite Gauss-Jordan Algorithm in Constructing Submatrices of QHF}
\label{sec:ConstructionViaInfiniteGaussJordanAlgorithm}
The infinite Gauss-Jordan algorithm implemented with RPS is extended by means of a reordering routine generating a chain of submatrices of a QHF of $C\in \RFM$. Let us call $\mathfrak{m}(C\!\!\mid_n):={\rm max}\{\maxs(C_k):\ 0\le k \le n\}$. As $L$ is in LRRF, the map of row-length of $L$ is injective and so $\maxs(L_{n+1})\not=\mathfrak{m}(L\!\!\mid_n)$. In view of the proof of Proposition \ref{proposition LRRF_3}, the definition of $\mu: \omega\mapsto \omega$ in (\ref{bijection}) indicates that row permutations must apply to the set of nonzero rows exclusively, since $\mu$ maps $J$ onto $J$. Meanwhile the position of zero rows should remain unchanged, since $\mu:W\mapsto W$ is the identity. As it is shown in this Section, these conditions prevent the grouping of the zero rows on the top of the permuted matrix, thus keeping the matrix in $\RFM$. 
\paragraph{Reordering Routine:}
\label{sec:ReorderingRoutine}
Rearrange non-zero rows of $L\!\!\mid_n$ in strictly increasing row-length, while leaving the position of zero rows unchanged. 

Choose a top submatrix of $L$, say $L\!\!\mid_{k_0}$. The reordering routine applied to $L\!\!\mid_{k_0}$ results in a matrix ${\mathcal{Q}}^{(k_0)}=P_{k_0}\cdot L\!\!\mid_{k_0}$, where $P_{k_0}$ is a permutation matrix. In the next step ${\mathcal{Q}}^{(k_0)}$ is augmented by the row $L_{k_0+1}$. If $L_{k_0+1}=\textbf{0}$, a new row is included. If $L_{k_0+1}\not=\textbf{0}$ and $\maxs(L_{k_0+1})<\mathfrak{m}(L\!\!\mid_{k_0})$ the reordering routine generates ${\mathcal{Q}}^{(k_0+1)}=P_{k_0+1}\cdot L\!\!\mid_{k_0+1}$ moving $L_{k_0+1}$ in a lower position. Meanwhile, the last nonzero row of ${\mathcal{Q}}^{(k_0)}$ of row-length $\mathfrak{m}({\mathcal{Q}}^{(k_0)})$ is placed in the highest row position of ${\mathcal{Q}}^{(k_0+1)}$, thus not belonging to ${\mathcal{Q}}^{(k_0+1)}\!\!\mid_{k_0}$. Thereafter a new row is included and the process continues at infinitum. 

The infinite sequence $({\mathcal{Q}}^{(k)})_{k\ge k_0}$ consists of matrices in QHF. Moreover, we can start applying the reordering routine to any top submatrix of $L$. Fixing $k$, a matrix ${\mathcal{Q}}^{(n)}\!\!\mid_k$ will change, if there is some $m$ with $m>n$ such that $L_m\not=\textbf{0}$ and $\maxs(L_m)<\mathfrak{m}({\mathcal{Q}}^{(n)}\!\!\mid_k)$.
After moving $L_m$ in a lower position its successive nonzero rows are placed in higher positions, meanwhile the positions of zero rows are fixed. It turns out that, if $k<n<m$, then
\begin{align}\label{rearr}
{\mathcal{Q}}^{(m)}\!\!\mid_k& ={\mathcal{Q}}^{(n)}\!\!\mid_k \Longleftrightarrow\mathfrak{m}({\mathcal{Q}}^{(m)}\!\!\mid_k)=\mathfrak{m}({\mathcal{Q}}^{(n)}\!\!\mid_k)\\
{\mathcal{Q}}^{(m)}\!\!\mid_k& \not={\mathcal{Q}}^{(n)}\!\!\mid_k \Longleftrightarrow\mathfrak{m}({\mathcal{Q}}^{(m)}\!\!\mid_k)<\mathfrak{m}({\mathcal{Q}}^{(n)}\!\!\mid_k)
\end{align}
Thus, if $k<n<m$, then $\mathfrak{m}({\mathcal{Q}}^{(m)}\!\!\mid_k)\le \mathfrak{m}({\mathcal{Q}}^{(n)}\!\!\mid_k)$\vspace{0.02in}. Let us apply the reordering routine to a concrete matrix $A=\left( \begin{array}{c} {\mathcal{Q}}^{(3)} \\ \hline  L_4 \end{array}\right)$\vspace{0.02in}, that is ${\mathcal{Q}}^{(3)}$ augmented by $L_4$. 
\begin{equation*}
\left( \begin{array}{c} {\mathcal{Q}}^{(3)} \\ \hline  L_4 \end{array}\right)=
\left( \begin{array}{ccccc} 
1 & 0 & 0 &  0 &... \\
0 & 0 & 0 &  0 &... \\
0 & 0 & 1 &  0 &... \\
0 & 0 & 0 &  0 &... \\
\hline \vspace{-0.10in}\\
0 & 1 & 0 & 0 &...\\
\end{array} \right)\longrightarrow
                   {\mathcal{Q}}^{(4)}= \left( \begin{array}{ccccc} 
                    1 & 0 & 0 &  0 &... \\
                    0 & 0 & 0 &  0 &... \\
                    0 & 1 & 0 &  0 &...\\
                    0 & 0 & 0 &  0 &... \\
                    0 & 0 & 1 &  0 &... \\
                    \end{array} \right)  
\end{equation*}
The resulting matrix ${\mathcal{Q}}^{(4)}$ is a QHF of $A$ such that
\[{\mathcal{Q}}^{(3)}={\mathcal{Q}}^{(3)}\!\!\mid_3= \left( \begin{array}{ccccc} 
                    1 & 0 & 0 &  0 &... \\
                    0 & 0 & 0 &  0 &... \\
                    0 & 0 & 1 &  0 &... \\
                    0 & 0 & 0 &  0 &... \end{array} \right) \not= \left( \begin{array}{ccccc} 
                                                                    1 & 0 & 0 &  0 &... \\
                                                                    0 & 0 & 0 &  0 &... \\
                                                                    0 & 1 & 0 &  0 &... \\
                                                                    0 & 0 & 0 &  0 &... \end{array} \right)=  {\mathcal{Q}}^{(4)}\!\!\mid_3\] 
and $\mathfrak{m}({\mathcal{Q}}^{(4)}\!\!\mid_3)\lneqq \mathfrak{m}({\mathcal{Q}}^{(3)}\!\!\mid_3)$. Furthermore, as the next example illustrates, the condition that the positions of zero rows must be fixed is sufficient to avoid grouping zero rows on the top of the matrix. 
\begin{example} \label{example grouping zero rows}{\rm
Let $A$ be the row-finite $\omega\times\omega$ matrix consisting of the same repeated row $A_n=(1,0,0,...)$ for all $n\in \omega$. The infinite Gauss-Jordan algorithm is implemented with row permutations grouping zero rows on the top part of ${\mathcal{Q}}^{(n)}$. This process gives rise to the following sequence of matrices: 
\[
{\mathcal{Q}}^{(1)}= \left( \begin{array}{cccc} 
0 & 0 & 0 & ... \\
1 & 0 & 0 &... \\
\end{array} \right) \longrightarrow
                   {\mathcal{Q}}^{(2)}= \left( \begin{array}{cccc} 
                    0 & 0 & 0 & ... \\
                    0 & 0 & 0 & ... \\
                    1 & 0 & 0 & ...\\
                    \end{array} \right)  \longrightarrow ...
\]
Upon algorithm completion, the only possible alternative assumption to an undecidable problem is a matrix of the form:
\[
H_1= \left( \begin{array}{cccc} 
0 & 0 & 0 & ... \\
0 & 0 & 0 &... \vspace{-0.07in}\\
.&.&.&... \vspace{-0.1in}\\
.&.&.&... \vspace{-0.1in}\\
.&.&.&... \\
\hline \vspace{-0.15in}\\
1 & 0 & 0 & ...
\end{array} \right)
\]
As the indexing set of zero rows is $\omega$, the row indexing set of $H_1$ is $\omega+1$, whence $H_1\not\in \RFM$. 
Nevertheless, by keeping the position of zero rows unchanged the outcome is the matrix:
\[
H_2= \left( \begin{array}{cccc} 
1 & 0 & 0 & ... \\
0 & 0 & 0 &... \\
0 & 0 & 0 &...\vspace{-0.05in} \\
.&.&.&... \vspace{-0.1in}\\
.&.&.&... \vspace{-0.1in}\\
.&.&.&...
\end{array} \right)
\]
Since $1+\omega=\omega$, it follows that $H_2\in \RFM$.}
\end{example}
The subsequent theorem shows that row permutations produced by the reordering routine do not affect a top submatrix of ${\mathcal{Q}}^{(k)}$ for large enough $k$.
\begin{theorem} \label{theorem row-rearrangement}
For every $k\in \omega$ there exists $\delta_k\in \omega$ with ${\delta}_{k}\ge k$ such that
\[ \forall n : n\ge \delta_k \Rightarrow  {\mathcal{Q}}^{(n)}\!\!\mid_k ={\mathcal{Q}}^{(\delta_k)}\!\!\mid_k\]
\end{theorem}
\begin{proof} 
On the contrary, let there is some $k\in \omega$ such that for every $\delta>k$
\begin{equation} \label{eq-contradiction2}
  \exists n_{\delta}: n_{\delta}> \delta \Rightarrow {\mathcal{Q}}^{(n_{\delta})}\!\!\mid_k \not={\mathcal{Q}}^{(\delta)}\!\!\mid_k
\end{equation} 
In view of (\ref{rearr}), the latter statement is equivalent to the following: There exists $k\in \omega$ such that for every $\delta>k $
\begin{equation} \label{eq-contradiction3}
  \exists n_{\delta}: n_{\delta}> \delta \Rightarrow \mathfrak{m}({\mathcal{Q}}^{(n_{\delta})}\!\!\mid_k)
  < \mathfrak{m}({\mathcal{Q}}^{(\delta)}\!\!\mid_k). 
\end{equation} 
A sequence of natural numbers is constructed as follows. Applying (\ref{eq-contradiction3}) with $\delta=k+1$, there is $n_{k+1}: n_{k+1}>k+1$ such that $\mathfrak{m}({\mathcal{Q}}^{(n_{k+1})}\!\!\mid_k)<\mathfrak{m}({\mathcal{Q}}^{(k+1)}\!\!\mid_k)$. Let us call $\delta_1={\rm min}\{n_{k+1}>k+1: \mathfrak{m}({\mathcal{Q}}^{(n_{k+1})}\!\!\mid_k)<\mathfrak{m}({\mathcal{Q}}^{(k+1)}\!\!\mid_k)$. As a direct consequence, $\mathfrak{m}({\mathcal{Q}}^{({\delta}_1)}\!\!\mid_k)<\mathfrak{m}({\mathcal{Q}}^{(k+1)}\!\!\mid_k)$.
Applying (\ref{eq-contradiction3}) with $\delta= \delta_{1}$ we have
\begin{equation} \label{eq-contradiction4}
  \exists n_{\delta_1}: n_{\delta_1}> \delta_1 \Rightarrow \mathfrak{m}({\mathcal{Q}}^{(n_{\delta_1})}\!\!\mid_k)<\mathfrak{m}({\mathcal{Q}}^{(\delta_1)}\!\!\mid_k)
\end{equation}
and we define $\delta_{2}={\rm min}\{n_{{\delta}_{1}}> \delta_{1}: \mathfrak{m}({\mathcal{Q}}^{(n_{\delta_1})}\!\!\mid_k)<\mathfrak{m}({\mathcal{Q}}^{(\delta_1)}\!\!\mid_k)\}$. Certainly $\mathfrak{m}({\mathcal{Q}}^{({\delta}_2)}\!\!\mid_k)<\mathfrak{m}({\mathcal{Q}}^{(\delta_1)}\!\!\mid_k)$. The process continues in this manner ad infinitum. Let us denote  $\delta_0:=k+1$. Inasmuch as $\delta_0<\delta_1<\delta_2<...$ and $\mathfrak{m}({\mathcal{Q}}^{(\delta_0)}\!\!\mid_k)> \mathfrak{m}({\mathcal{Q}}^{(\delta_1)}\!\!\mid_k)>\mathfrak{m}({\mathcal{Q}}^{(\delta_2)}\!\!\mid_k)>... $, a strictly decreasing infinite sequence in $\omega$ is constructed, which is impossible.
\end{proof}
As the rows of ${\mathcal{Q}}^{(n)}\!\!\mid_k$, $n\ge \delta_k$, remain invariant under further row permutations by means of the reordering routine, we can write ${\mathcal{Q}}^{(\delta_{k})}\!\!\mid_k=(H_n)_{0\le n\le k}$ for all $k\in \omega$ and define \vspace{0.03in}$H=\bigcup_{k\in \omega}{\mathcal{Q}}^{(\delta_{k})}\!\!\mid_k=\bigcup_{k\in \omega} \{(n,H_n): 0\le n\le k\}$\vspace{0.03in}. Certainly $H\in \RFM$. Moreover, taking into account that $L$ is a LRRF of $C$ and $H$ has the same sets of nonzero and zero rows as $L$, it follows that $\RS(L)=\RS(H)$ and $\nul(L)=\nul(H)$. Thus Proposition \ref{proposition_3.1} implies $H\sim L$. Also, as $H$ is in LRRF of strictly increasing sequence of row-length, $H$ is a QHF of $C$. As a conclusion, applying the reordering routine to $L$ we deduce:

\begin{corollary} \label{corollary row-rearrangement_1} For every $n\in \omega$ there is $\delta_{n}\in\omega$ with $\delta_{n}\ge n$ such that ${\mathcal{Q}}^{(\delta_n)}\!\!\mid_n\sqsubset H$. Thus a chain of top submatrices of $H$ is generated satisfying:
\begin{equation} \label{submatrices of QHF}
{\mathcal{Q}}^{(\delta_0)}\!\!\mid_0 \sqsubset {\mathcal{Q}}^{(\delta_1)}\!\!\mid_1\sqsubset ... \sqsubset {\mathcal{Q}}^{(\delta_n)}\!\!\mid_n\sqsubset ...\sqsubset H
\end{equation}
\end{corollary}
\textbf{The Extended Algorithm}. The Infinite Gauss-Jordan algorithm is supplemented by including the reordering routine as the forth and final part of the $n$-th stage of the process. In this setting, instead of applying the reordering routine directly to submatrices of $L$, it now applies to matrices of the form ${\mathcal{L}}^{(n)}$, which may not be submatrices of $L$. However, as we show in what follows, given any desired number of rows, say $n$, for a large enough $N\ge n$, it turns out that ${\mathcal{L}}^{(N)}\!\!\mid_n \sqsubset L$ and simultaneously ${\mathcal{Q}}^{(N)}\!\!\mid_n \sqsubset H$, where $H$ is a QHF of $C$.
Let $\delta_n, \epsilon_n$ be as in Corollary \ref{corollary row-rearrangement_1} and Proposition \ref{prop-GJ2}, respectively. For any $n\in \omega$, let $N=\max\{\epsilon_n,\delta_n\}$. At the $N$-th stage, as $N\ge \epsilon_n$, the algorithm gives ${\mathcal{L}}^{(N)}$ such that ${\mathcal{L}}^{(N)}\!\!\mid_n ={\mathcal{L}}^{(\epsilon_n)}\!\!\mid_n= L\!\!\mid_n \sqsubset L$. At the $\epsilon_N$-stage, on account of $N\ge \delta_n$ and ${\mathcal{L}}^{(\epsilon_N)}_N=L\!\!\mid_N$, the reordering routine applied to ${\mathcal{L}}^{(\epsilon_N)}_N$ gives ${\mathcal{Q}}^{(N)}\!\!\mid_n ={\mathcal{Q}}^{(\delta_n)}\!\!\mid_n\sqsubset H$.

\textbf{An Alternative Formulation.} In order to construct a chain of submatrices of a QHF of $C$, the uniqueness of ${\mathcal{Q}}^{(N)}\!\!\mid_n$ provides the following alternative. Any Gauss-Jordan elimination scheme, including the latest software innovations such as partial pivoting techniques, when implemented with RPS on a top submatrix $C\!\!\mid_N$ of $C$ gives a matrix ${\mathcal{Q}}^{(N)}$ in LRREF. If $N$ is sufficiently large ${\mathcal{Q}}^{(N)}$ contains a submatrix ${\mathcal{Q}}^{(N)}\!\!\mid_n$ of a QHF of $C$, since the set of all zero rows grouped on the top of ${\mathcal{Q}}^{(N)}$ is finite. Thereafter, the infinite Gauss-Jordan algorithm, supplemented by the reordering routine, is applied to the matrix ${\mathcal{Q}}^{(N)}$ augmented by successive rows of $C$, thus generating a chain of submatrices of a QHF of $C$, as described by (\ref{submatrices of QHF}).
\subsection{Examples on the Construction of QHF}
\label{sec:Examples}
The infinite Gauss-Jordan algorithm is implemented with RPS to row reduce two additional concrete examples. Example \ref{example2} was treated by recursion in \cite{Fu:Th} and the same results are recovered here. Special attention is paid to the conditions which make the recurrence (\ref{rec-Hermite}) applicable. In Example \ref{example3} the matrix representation of a partial differential operator is considered. In this event, condition (\ref{cond-Hermite2}) is not satisfied, making impossible a direct application of (\ref{rec-Hermite}).
\begin{example} \label{example2} {\rm
Following Fulkerson, we define the row-finite $\omega\times\omega$ matrix:
\begin{equation}
\label{firstmtr}
A= \left( \begin{array}{cccccccccccc} 
0 & 0 & 1 & 1 & 0 & 0 & 0 & 0 & 0 & 0 & ... \\
0 & 0 & 0 & 0 & 0 & 0 & 0 & 0 & 0 & 0 &... \\
0 & 0 & 0 & 1 & 0 & 1 & 1 & 0 & 0 & 0 &... \\
0 & 0 & 1 & 3 & 0 & 2 & 2 & 0 & 0 & 0 &... \\
0 & 0 & 0 & 1 & 0 & 0 & 1 & 0 & 1 & 1 &... \\
0 & 0 & 1 & 5 & 0 & 1 & 4 & 0 & 3 & 3 &...\vspace{-0.07in}\\
.&.&.&.&.&.&.&.&.&.&...
\end{array} \right)
\end{equation} 
The first three rows of $A$ are given by $A_{0}=\mbox{\boldmath $\mathrm{e}$}_{2}+\mbox{\boldmath $\mathrm{e}$}_{3}, A_{1}=\textbf{0}, A_{2}=\mbox{\boldmath $\mathrm{e}$}_{3}+\mbox{\boldmath $\mathrm{e}$}_{5}+\mbox{\boldmath $\mathrm{e}$}_{6}$, and the remaining rows are generated by the formulas:
\begin{equation*} \begin{array}{lll}
A_{2n+1}&=&\displaystyle (n+1)A_{2n}+\sum_{i=0}^{n-1}A_{2i}, \ \ \ {\rm for} \ n\ge 1    \vspace{0.05in} \\
A_{2n}&=&\mbox{\boldmath $\mathrm{e}$}_{3}+\mbox{\boldmath $\mathrm{e}$}_{6}+\mbox{\boldmath $\mathrm{e}$}_{3n+2}+\mbox{\boldmath $\mathrm{e}$}_{3(n+1)}, \ \ \ {\rm for} \ n\ge 2   
\end{array}
\end{equation*}
Let $\alpha\in\Hom$ be the linear mapping induced by $A$, thus $A=[\alpha]^{\mbox{\scriptsize \boldmath $\mathrm{E}$}}_{\mbox{\scriptsize \boldmath $\mathrm{E}$}}$. The implementation of the infinite Gauss-Jordan algorithm with RPS, including the reordering part, results in a QHF of $A$, say $H$. By applying the same sequence of row operations to $\mbox{\boldmath $\mathrm{E}$}$, which have occurred in the row reduction of $A$ to $H$, the domain basis $\mbox{\boldmath $\mathrm{L}$}=\{\mbox{\boldmath $\mathrm{L}$}_n\}_{n\in \omega}$ is derived. In what follows, the elements of $\mbox{\boldmath $\mathrm{L}$}$ are appended next to corresponding rows of $H$ occupying the right-end column:  
\[H\!\!=\!\!\left(
\begin{array}{llrllrlllllllll}
 0 & 0 & 1 & 1 & 0 & 0 & 0 & 0 & 0 & 0 & 0 & 0 &... \\
 0 & 0 & 0 & 0 & 0 & 0 & 0 & 0 & 0 & 0 & 0 & 0 &...\\
 0 & 0 & -1 & 0 & 0 & 1 & 1 & 0 & 0 & 0 & 0 & 0 &  ...\\
 0 & 0 & 0 & 0 & 0 & 0 & 0 & 0 & 0 & 0 & 0 & 0 & ...\\
 0 & 0 & 0 & 0 & 0 &-1 & 0 & 0 & 1 & 1 & 0 & 0 &... \\
 0 & 0 & 0 & 0 & 0 & 0 & 0 & 0 & 0 & 0 & 0 & 0 &...\\
 0 & 0 & 0 & 0 & 0 &-1 & 0 & 0 & 0 & 0 & 1 & 1 & ...\\
  .&.&.&.&.&.&.&.&.&.&.&.&...
\end{array}
\right)
\left|\begin{array}{l}
\mbox{\boldmath $\mathrm{e}$}_{0} \\ 
\mbox{\boldmath $\mathrm{e}$}_{1} \\ 
\mbox{\boldmath $\mathrm{e}$}_{2}-\mbox{\boldmath $\mathrm{e}$}_{0}\\ 
\mbox{\boldmath $\mathrm{e}$}_{3}-2\mbox{\boldmath $\mathrm{e}$}_{2}-\mbox{\boldmath $\mathrm{e}$}_{0} \\ 
\mbox{\boldmath $\mathrm{e}$}_{4}-\mbox{\boldmath $\mathrm{e}$}_{2}\\ 
\mbox{\boldmath $\mathrm{e}$}_{5}-3\mbox{\boldmath $\mathrm{e}$}_{4}-\mbox{\boldmath $\mathrm{e}$}_{2}-\mbox{\boldmath $\mathrm{e}$}_{0} \\
\mbox{\boldmath $\mathrm{e}$}_{6}-\mbox{\boldmath $\mathrm{e}$}_{2} \\  
...
\end{array}\right. 
\]
The defining relation of $(A_{2n+1})_{n\ge 1}$ is $A_{2n+1}-(n+1)A_{2n}-\sum_{i=0}^{n-1}A_{2i}=\textbf{0}$. On account of $\alpha(\mbox{\boldmath $\mathrm{e}$}_{k})=A_k$, we deduce the following alternative calculation of a basis of $\NS(A)$: 
$\mbox{\boldmath $\mathrm{L}$}_{1}=\mbox{\boldmath $\mathrm{e}$}_{1}, \ \mbox{\boldmath $\mathrm{L}$}_{2n+1}= \mbox{\boldmath $\mathrm{e}$}_{2n+1}-(n+1)\mbox{\boldmath $\mathrm{e}$}_{2n}-\sum_{i=0}^{n-1}\mbox{\boldmath $\mathrm{e}$}_{2i}, \ n\ge 1$;
in full accord with the result derived by the infinite Gauss-Jordan algorithm. The non-singular matrix of passage from $\mbox{\boldmath $\mathrm{E}$}$ to $\mbox{\boldmath $\mathrm{L}$}$ is given by
\[Q=\left(
\begin{array}{rlrlrrllll}
 1 & 0 & 0 & 0 & 0 & 0 & 0 & ... \\
 0 & 1 & 0 & 0 & 0 & 0 & 0 &...\\
 -1 & 0 & 1 & 0 & 0 & 0  & 0 & ...\\
 -1 & 0 & -2 & 1 & 0 & 0 & 0 & ...\\
 0 & 0 & -1 & 0 & 1 & 0 & 0 & ... \\
 -1 & 0 &-1 & 0 & -3 & 1 & 0 & ...\\
 0 & 0 & -1 & 0 & 0 & 0 & 1 & ... \\
  .&.&.&.&.&.&.&...
\end{array}\right)
\]
consisting of the coefficients of the vectors in $\mbox{\boldmath $\mathrm{L}$}$. It is easily verified that $H=Q \cdot A$ and $A\sim H$. It is worth noting the significantly less amount of automatic work in calculating simultaneously the above shown parts of $H,Q$ compared to the corresponding amount of manual work in \cite{Fu:Th}.

Let us next verify that $A$ meets the condition (\ref{cond-Hermite2}) to (\ref{rec-Hermite}). Since the map of row-length of $A$ is not injective, Theorem \ref{theorem complete-basis} implies that the set of nonzero rows of $A$ is not a complete basis of row-length of $\RS(A)$. However, the defining relation of $A_{2n+1}$ entails that every odd row of $A$ for $n\geq 1$ is a linear combination of the even rows $\{A_{2i}\}_{0\le i\leq n}$ of $A$. Thus $\Span(\{A_{2n}\}_{n\in \omega})={\rm RS}(A)$. Furthermore, as $\maxs(A_{2n})=3(n+1)$ for all $n\in \omega$, we infer that $\{\rho_{2n}\}_{n\in \omega}$ is strictly increasing. In particular, the set of length-equivalence classes of $\RS(A)$ is $\Class =\{ \Class_{3(n+1)}\}_{n\in \omega}$ and $A_{2n}\in \Class_{3(n+1)}$ for all $n\in \omega$. Corollary \ref{corollary_1.4} implies that the set of even rows of $A$ is a complete basis of row-length of ${\RS}(A)$. As a conclusion $\{A_{2n}\}_{n\in \omega}$, is available in advance and so Fulkerson's recurrence is directly applicable. By setting $\mbox{\boldmath$\mathrm{A}$}_{\rho_{2n}}=A_{2n}$ in (\ref{rec-Hermite}) we obtain
\begin{equation*} \begin{array}{l} H_0=A_0, H_2=A_2-A_0,..., H_{2n}=A_{2n}-H_2-H_0=A_{2n}-A_2, \ n\ge 2, \ {\rm or} \\ 
H_0=\mbox{\boldmath $\mathrm{e}$}_{2}+\mbox{\boldmath $\mathrm{e}$}_{3},  \ H_2=-\mbox{\boldmath $\mathrm{e}$}_2+\mbox{\boldmath $\mathrm{e}$}_5+\mbox{\boldmath $\mathrm{e}$}_6,...,H_{2n}=-\mbox{\boldmath $\mathrm{e}$}_2+\mbox{\boldmath $\mathrm{e}$}_{6}+\mbox{\boldmath $\mathrm{e}$}_{3n+2}+\mbox{\boldmath $\mathrm{e}$}_{3(n+1)}
\end{array}
\end{equation*}
which coincide with the corresponding nonzero rows of $H$.}
\end{example} 
In the sequel, we shall adopt the notation $I=\omega\times \omega$ coupled with the standard well ordering on $I$ defined by
\[ (i,j)\prec_1 (n,m) \Longleftrightarrow \ \mbox{$i+j<n+m$}\ \rm{or} \ (\mbox{$i+j=n+m$} \ \rm{and} \ \mbox{$j<m$}). \]
The elements of $I$ are listed, relative to, $\prec_1$ as follows: $(0,0)\prec_1 (1,0)\prec_1 (0,1)\prec_1 (2,0)\prec_1 (1,1)\prec_1 (0,2)\prec_1 (3,0),...$ and the ordinal number associated with $(I, \prec_1 )$ is $\omega$. Let us call ${\mathcal{Y}}$ the standard basis of the space of bivariate polynomials equipped with the well ordering $\prec_1$. Formally ${\mathcal{Y}}= \{1, x, y, x^{2}, xy, y^{2}, x^{3}, x^{2}y, xy^{2},y^{3},...\}$; in which the terms are listed with respect to $\prec_1$. We shall also use the following well ordering on $I$: 
\begin{equation*} 
(i,j)\prec_2 (n,m) \Longleftrightarrow \left\{\begin{array}{l} \mbox{$i+j<n+m$} \ \ {\rm or} \\
                                                           \mbox{$i+j=n+m$}  \ \rm{and} \ \mbox{$n+m$} \  \rm{is \  odd} \ \rm{and} \  \mbox{$i<n$}  \ \ {\rm or} \\
                                                           \mbox{$i+j=n+m$}  \ \rm{and} \  \mbox{$n+m$} \  \rm{is \ even} \ \rm{and} \ \mbox{$j<m$} 
                                              
   \end{array} \right.
\end{equation*}
The elements of $I$ are listed with respect to $\prec_2$ as follows: $(0,0)\prec_2 (0,1)\prec_2 (1,0)\prec_2 (2,0)\prec_2 (1,1)\prec_2 (0,2)\prec_2 (0,3),...$ and the ordinality of $(I, \prec_2)$ is $\omega$. Let us call ${\mathcal{X}}$ the standard basis of bivariate polynomials, but now equipped with the ordering $\prec_2$, that is ${\mathcal{X}}= \{1, y, x, x^{2}, xy, y^{2}, y^{3}, y^{2}x, yx^{2},x^{3},x^4...\}$.
\begin{example}\label{example3} {\rm   Let ${\mathcal{P}}$ be the space of bivariate polynomials. Let also $D$ be the partial differential operator
    $D:= (x^{2}+xy+y^{2}) \frac{\partial^{2}}{\partial x \partial y}+xy(\frac{\partial}{\partial x}+\frac{\partial}{\partial y})$
considered as an endomorphism of ${\mathcal{P}}$. The range of $D$ is spanned by:
\begin{equation*} \label{eq-ex21}
D(x^{n}y^{m})= nmx^{n+1}y^{m-1}+ nmx^{n}y^{m}+nmx^{n-1}y^{m+1}+mx^{n+1}y^{m}+nx^{n}y^{m+1}
\end{equation*}
The matrix representation of $D$, relative to $({\mathcal{X}}, {\mathcal{Y}})$, is displayed in what follows:
\begin{equation} \label{matrix-M} 
M\! =\!\left(
 \begin{array}{lllllllllllllllll}
 0 & 0 & 0 & 0 & 0 & 0 & 0 & 0 & 0 & 0 & 0 & 0 & 0 & 0 &\!\! ...\\
 0 & 0 & 0 & 0 & 1 & 0 & 0 & 0 & 0 & 0 & 0 & 0 & 0 & 0 &\!\! ...\\
 0 & 0 & 0 & 0 & 1 & 0 & 0 & 0 & 0 & 0 & 0 & 0 & 0 & 0 &\!\! ...\\
 0 & 0 & 0 & 0 & 0 & 0 & 0 & 2 & 0 & 0 & 0 & 0 & 0 & 0 &\!\! ...\\
 0 & 0 & 0 & 1 & 1 & 1 & 0 & 1 & 1 & 0 & 0 & 0 & 0 & 0 &\!\! ...\\
 0 & 0 & 0 & 0 & 0 & 0 & 0 & 0 & 2 & 0 & 0 & 0 & 0 & 0 &\!\! ...\\
 0 & 0 & 0 & 0 & 0 & 0 & 0 & 0 & 0 & 0 & 0 & 0 & 0 & 3 &\!\! ...\\
 0 & 0 & 0 & 0 & 0 & 0 & 0 & 2 & 2 & 2 & 0 & 0 & 2 & 1 &\!\! ...\\
 0 & 0 & 0 & 0 & 0 & 0 & 2 & 2 & 2 & 0 & 0 & 1 & 2 & 0 &\!\! ...\\
 0 & 0 & 0 & 0 & 0 & 0 & 0 & 0 & 0 & 0 & 0 & 3 & 0 & 0 &\!\! ...\vspace{-0.05in}\\
 .&.&.&.&.&.&.&.&.&.&.&.&.&.&\!\! ...\\
\end{array}
\right)
\end{equation}
The infinite Gauss-Jordan algorithm, implemented with RPS and including the reordering part, results in a QHF of $M$:
\begin{equation} \label{matrix-H}
H=\left(
\begin{array}{llllllrlllllllll}
 0 & 0 & 0 & 0 & 0 & 0 &\!\! 0 & 0 & 0 & 0 & 0 & 0 & 0 & 0 & \!\!\!\! ...\\
 0 & 0 & 0 & 0 & 1 & 0 &\!\! 0 & 0 & 0 & 0 & 0 & 0 & 0 & 0 & \!\!\!\! ...\\
 0 & 0 & 0 & 0 & 0 & 0 &\!\! 0 & 0 & 0 & 0 & 0 & 0 & 0 & 0 &\!\!\!\! ... \\
 0 & 0 & 0 & 1 & 0 & 1 &\!\! 0 & 0 & 0 & 0 & 0 & 0 & 0 & 0 & \!\!\!\! ...\\
 0 & 0 & 0 & 0 & 0 & 0 &\!\! 0 & 1 & 0 & 0 & 0 & 0 & 0 & 0 & \!\!\!\! ...\\
 0 & 0 & 0 & 0 & 0 & 0 &\!\! 0 & 0 & 1 & 0 & 0 & 0 & 0 & 0 &\!\!\!\! ...\\
 0 & 0 & 0 & 0 & 0 & 0 &\!\! -1 & 0 & 0 & 1 & 0 & 0 & 0 & 0&\!\!\!\! ...\\
 0 & 0 & 0 & 0 & 0 & 0 &\!\! 0 & 0 & 0 & 0 & 0 & 1 & 0 & 0 &\!\!\!\! ...\\
 0 & 0 & 0 & 0 & 0 & 0 &\!\! 1 & 0 & 0 & 0 & 0 & 0 & 1 & 0 &\!\!\!\! ...\\
 0 & 0 & 0 & 0 & 0 & 0 &\!\! 0 & 0 & 0 & 0 & 0 & 0 & 0 & 1 & \!\!\!\! ...\vspace{-0.07in}\\
.&.&.&.&.&.&.&.&.&.&.&.&.&.&\!\!\!\! ...
\end{array}\right)
\end{equation}
By applying the same sequence of row operations to the identity $\omega\times\omega$ matrix $\textbf{I}$, the derived non-singular matrix of passage from ${\mathcal{X}}$ to $\mbox{\boldmath $\mathrm{L}$}$ is given by
\[Q=\left(
\begin{array}{lrlrlrrrrrr}
1 & 0 & 0 & 0 & 0 & 0 & 0 & 0 & 0 & 0 &... \vspace{0.02in}\\ 
0 & 1 & 0 & 0 & 0 & 0 & 0 & 0 & 0 & 0 & ... \vspace{0.02in}\\
0 & -1 & 1 & 0 & 0 & 0 & 0 & 0 & 0 & 0 & ... \vspace{0.02in}\\
0 & -1 & 0 & -\frac{1}{2} & 1 & -\frac{1}{2} & 0 & 0 & 0 & 0 & ... \vspace{0.02in}\\
0 & 0 & 0 & \frac{1}{2} & 0 & 0 & 0 & 0 & 0 & 0 & ... \vspace{0.02in}\\
0 & 0 & 0 & 0 & 0 & \frac{1}{2} & 0 & 0 & 0 & 0 & ... \vspace{0.02in}\\
0 & 0 & 0 & 0 & 0 & 0 & -\frac{1}{6} & \frac{1}{2} & -\frac{1}{2} & \frac{1}{6}&... \vspace{0.02in}\\
0 & 0 & 0 & 0 & 0 & 0 & 0 & 0 & 0 & \frac{1}{3}& ... \vspace{0.02in}\\
0 & 0 & 0 & -\frac{1}{2} & 0 & -\frac{1}{2} & 0 & 0 & \frac{1}{2} & -\frac{1}{6}&... \vspace{0.02in}\\
0 & 0 & 0 & 0 & 0 & 0 & \frac{1}{3} & 0 & 0 & 0 & ... \vspace{-0.07in} \\
.&.&.&.&.&.&.&.&.&.&...
\end{array}
\right)
\] 
It is easily verified that $Q\cdot M=H$ and $M\sim H$.

As an illustration of the results of Subsection \ref{sec:ConstructionViaInfiniteGaussJordanAlgorithm} concerning the extended infinite Gauss-Jordan algorithm, it is noted that:
i) The smallest $N$ such that  ${\mathcal{L}}^{(N)}\!\!\mid_6 \sqsubset L$ and ${\mathcal{Q}}^{(N)}\!\!\mid_6 \sqsubset H$ is $N=\delta_6=\epsilon_6=9$. ii) ${\mathcal{Q}}^{(9)}\!\!\mid_6 \sqsubset{\mathcal{Q}}^{(9)}\!\!\mid_7 \sqsubset{\mathcal{Q}}^{(9)}\!\!\mid_8 \sqsubset{\mathcal{Q}}^{(9)}\!\!\mid_9 \sqsubset H$, thus $\delta_9=\epsilon_9=9$ and $H|_9={\mathcal{Q}}^{(9)}$, as displayed in (\ref{matrix-H}). This justifies the size of the initial matrix $M|_9$ chosen to display the results of this example.

Since $\nul(M)=2$, row permutations can be applied to $L$ so as to group zero rows on the top of the matrix, thus generating the LRREF of $M$. This same result is obtained by applying any Gauss-Jordan elimination routine with RPS, directly to the top submatrix $M|_9$ of $M$.

We observe that the map of row-length associated with $M$ is not injective. By comparing (\ref{matrix-M}) with (\ref{matrix-H}), we further notice the appearance of new row lengths in (\ref{matrix-H}) occupying the columns of $H$ of index $5, 9,...$. Unlike the previous example, it turns out that no subset of the set of the nonzero rows of $M$, is a complete basis of  row-length of $\RS(M)$. Therefore, the condition (\ref{cond-Hermite2}) is not fulfilled, that is to say, the recurrence (\ref{rec-Hermite}) is not directly applicable to $M$.}
\end{example}
\section{Infinite Systems of Linear Equations}
\label{sec:InfiniteSystemsOfLinearEquations}
In connection with the works of Toeplitz~\cite{To:Auf} and Fulkerson~\cite{Fu:Th}, the infinite Gauss-Jordan algorithm makes it possible to treat infinite linear systems with coefficient matrix in $\RFM$ as finite linear systems. The problem of the existence and uniqueness of solutions of linear systems of the above-mentioned type is treated in a purely algebraic manner along with the full construction of the general solution of the system. In this setting, the coefficient matrix is used as a left operator on the space ${\mathcal{F}}^{\omega}$ (or ${\mathcal{F}}^{\infty}$) of infinite sequences. As matrix multiplication requires, the elements of ${\mathcal{F}}^{\omega}$ are represented by column vectors and the induced linear mapping is given by $\alpha: \xi\mapsto \alpha(\xi)= A\cdot\xi$, thus being an endomorphism of ${\mathcal{F}}^{\omega}$.

Let $A\in \RFM$ and $x,c\in {\mathcal{F}}^{\omega}$. The infinite linear system is of the form:
\begin{equation} \label{initial-system}
A\cdot x= c.
\end{equation}
The extended Gauss-Jordan elimination algorithm, implemented with RPS and including the reordering part, results in a QHF of $A$, say $H$. As $A\sim H$, there is a non-singular $\omega\times\omega$ matrix, say $Q$, such that $Q\cdot A=H$. On account of
\[
A\cdot x = c \Longleftrightarrow  Q\cdot (A \cdot x)= Q\cdot c \Longleftrightarrow (Q\cdot A) \cdot x= Q\cdot c \Longleftrightarrow H\cdot x= Q\cdot c
\]
the system takes the equivalent form:
\begin{equation} \label{equivalent-system}
H\cdot x= Q\cdot c.
\end{equation}
In the context of the results of this paper, a simple alternative form of the general solution of (\ref{initial-system}) is formulated in what follows.

\subsection{Homogeneous Solution}
\label{sec:HomogeneousSolution} 
Let $W,J$ be the indexing sets of zero and nonzero rows of $H=(h_{ki})_{(k,i)\in \omega\times\omega}$, respectively, and $\{\rho_j\}_{j\in J}$ be the strictly increasing sequence of row-length of $H$. Let $\defic(A)$ be the codimension of $\RS(A)$, called the \emph{deficiency or defect} of $A$. Since the nonzero rows of $H$, form a sequence of complete row-length (Hermite basis), Corollary \ref{Corollary_1.5} (ii) implies that the subset of $\mbox{\boldmath $\mathrm{E}$}$ indexed by $\omega\setminus \Right(H)$ spans a complementary space of $\RS(A)$, that is $\defic(A)=\card(\omega\setminus \Right(H))$.
If $c=\textbf{0}$, in view of (\ref{equivalent-system}), the homogeneous system (\ref{initial-system}) is equivalent to the system $H\cdot x= 0$.
The sequence $x_{\mathcal {H}}=(\lambda_m)^T_{m\in \omega}$ in ${\mathcal{F}}^{\omega}$, where the ``$T$" stands for transposition, is defined by
\[
  \lambda_m=\left\{\begin{array}{ccl} t_m  & ({\rm free\ variable})  & {\rm if} \ m\in \omega\setminus \Right(H) \\
  \\
\displaystyle                     -\sum^{\rho_i-1}_{\substack{ k=0 \\ k\not\in \Right(H)}}  h_{ik} t_k &  & {\rm if } \ m=\rho_i\in \Right(H)      
\end{array}\right.
\]
whereas the terms $t_k$ in the sum of the second branch were previously chosen arbitrarily in the first branch. Since $H$ is in QHF: $h_{ik}=0$ for $k\in \Right(H)$ and $k<\rho_i$, $h_{i\rho_i}=1$ and $h_{ik}=0$ for $k>\rho_i$. Thus for every $i\in J$
\[
\displaystyle H_i\cdot x_{\mathcal {H}}=\sum^{\rho_i}_{k=0} h_{ik}\lambda_k =\sum^{\rho_i-1}_{k=0} h_{ik}\lambda_k +h_{i\rho_i}\lambda_{\rho_i}=\sum^{\rho_i-1}_{\substack{ k=0 \\ k\not\in \Right(H)}}  h_{ik} t_k-\sum^{\rho_i-1}_{\substack{ k=0 \\ k\not\in \Right(H)}}  h_{ik} t_k=0.          
\] 
Also $H_i\cdot x_{\mathcal {H}}=0$ for all $i\in W$ and so $H\cdot x_{\mathcal {H}}=\textbf{0}$. Consequently $x_{\mathcal {H}}$ is the homogeneous solution of (\ref{initial-system}). Certainly the number of variables in $x_{\mathcal {H}}$ equals the $\defic(A)$. Thus nontrivial homogeneous solutions exist if and only if $\defic(A)>0$.
\subsection{General Solution}
\label{sec:GeneralSolution}
In view of (\ref{equivalent-system}), let us call $k:=Q\cdot c\in {\mathcal{F}}^{\omega}$. Then the non-homogeneous system (\ref{initial-system}) is equivalent to 
\begin{equation} \label{nonhomogeneous-system}
H\cdot x= k.
\end{equation}
Since $H_w=\textbf{0}$ for all $w\in W$, it follows that $H_w\cdot x=0$ for all $x\in {\mathcal{F}}^{\omega}$ and $w\in W$. On account of $H_w\cdot x=k_w$ a necessary condition for the existence of solutions of (\ref{nonhomogeneous-system}) is 
\begin{equation} \label{condition-particular-solution}
k_w=0, \ \  {\rm for \ all} \ \ w\in W.
\end{equation}
If $J=\omega\setminus W$, without loss of generality, we may assume that $J=\{j_0,j_1,...\}$ with $j_0<j_1<...$. Let us define 
\begin{equation} \label{particular-solution}
\begin{array} {llll}
x_{\mathcal {P}}=(0,0,...,0,\!\!\!\!\!& k_{j_0},0,0,...,0,\!\!\!\!\!& k_{j_1},0,...,0,\!\!\!\!\!&k_{j_i},0,0,...)^T \\
               & \uparrow                  & \uparrow                          & \uparrow        \\
               & \rho_0                    & \rho_1                            & \rho_i
               \end{array}  
\end{equation}
where $k_{j_i}$ has the position $\rho_i$. As $h_{i\rho_j}=0$ for $j<i$, $h_{i\rho_i}=1$ and  $h_{im}=0$ for all $m>\rho_i$, we deduce
\[
H_i\cdot x_{\mathcal {P}}=h_{i\rho_0}k_{j_0}+...+h_{i\rho_i}k_i= h_{i\rho_i}k_i=k_i                              
\]
for all $i\in J$. Assuming that $k_w=0$ for all $w\in W$ in (\ref{nonhomogeneous-system}), then $x=x_{\mathcal {P}}$ satisfies (\ref{nonhomogeneous-system}) and so $x_{\mathcal {P}}$ in (\ref{particular-solution}) is a particular solution of (\ref{initial-system}). Accordingly, condition (\ref{condition-particular-solution}) is also sufficient. This means that the system (\ref{initial-system}) is consistent if and only if $k_w=0$ for all $w\in W$. The general solution $x=(x_m)^T_{m\in \omega}$ of (\ref{initial-system}) is $x=x_{\mathcal {P}}+x_{\mathcal {H}}$, that is
\begin{equation} \label{general-solution}
            x_m=\left\{\begin{array}{ccl} t_m  & ({\rm free\ variable}) & {\rm if} \  m\in \omega\setminus \Right(H) \\
\\
\displaystyle                      k_{j_i}-\sum^{\rho_i-1}_{\substack{ k=0 \\ k\not\in \Right(H)}}h_{mk}t_k  &  & {\rm if } \ m=\rho_i\in \Right(H)      
\end{array}\right.    
\end{equation}

\subsection{Examples on Infinite Linear Systems}
\label{sec:ExamplesOnInfiniteSystems}
The homogeneous solutions of corresponding infinite linear systems with coefficient matrices defined in (\ref{ex1.matr2}), (\ref{firstmtr}) and (\ref{matrix-M}) are respectively given by
\[\begin{split}x_{\mathcal {H}}=& (t_0,-t_0, t_0,...)^T\\
               x_{\mathcal {H}}=& (t_0,t_1,t_2,-t_2,0,t_3, t_2-t_3,0,t_4,t_3-t_4, t_5, t_3-t_5,...)^T\\
               x_{\mathcal {H}}=& (t_0,t_1,t_2,t_3,0,-t_3,t_4,0,0,t_4,t_5,-t_4,0,...)^T,
\end{split}
\]
where $t_0,t_1,t_2,...$ are free variables in $\mathcal{F}$.

Using the notation of Example \ref{example1}, the non-homogeneous infinite linear system $C\cdot x=s$ with coefficient matrix defined in (\ref{ex1.matr2}) is consistent for any right-hand-side sequence $s=(s_i)^T_{i\in \omega}\in {\mathcal{F}}^{\omega}$, since $H$ in (\ref{eq-ex2.matr2}) has no zero rows. The components of $k=Q\cdot s$ are given below:
\[k_0=s_0, k_1=-s_0+s_1, k_2=s_0-s_1+s_2,... \]
In view of (\ref{particular-solution}), a particular solution of the system is $x_{\mathcal {P}}=(0,k_0,k_1,k_2,...)$ and so the general solution is of the form 
\[x=(t_0,s_0-t_0, -s_0+s_1+t_0, s_0-s_1+s_2+t_0...)^T,\]
where $t_0$ is a free variable in $\mathcal{F}$.

Following the notation of Example \ref{example2}, let us finally consider the non-homogeneous system $A\cdot x=c$ with coefficient matrix $A$ defined in (\ref{firstmtr}). Let $c=(c_i)^T_{i\in \omega}\in {\mathcal{F}}^{\omega}$. The components of $k=Q\cdot c$ are given by
\[\begin{split} k_0&=c_0, k_1=c_1, k_2=-c_0+c_2, k_3=-c_0-2c_2+c_3, k_4=-c_2+c_4,\\
                k_5&=-c_0-c_2-3c_4+c_5, k_6=-c_2+c_6,...  
                \end{split}
                \]
On account of $k_{2n+1}=0$ for $n\in \omega$, since $H_{2n+1}=0$ for all $n\in \omega$, we conclude that
\[c_1=0, c_3=c_0+2c_2, c_5=c_0+c_2+3c_4,... \]
Thus the form of $c$ so as the non-homogeneous system becomes consistent must be $c=(c_0, 0, c_2, c_0+c_2, c_4, c_0+c_2+3c_4, c_6,...)^{T}$, where $c_0,c_2,c_4,...$ are free variables in $\mathcal{F}$. According to (\ref{particular-solution}), a particular solution of the system is of the form:
\[x_{\mathcal {P}} =(0, 0, 0, c_0, 0, 0,-c_0+c_2, 0, 0, -c_2+c_4, 0, -c_2+c_6,...)^{T}. \]
Finally, in view of (\ref{general-solution}) the general solution of this system is given by
\[\begin{split}
   x=(&t_0,t_1,t_2,c_0-t_2,0,t_3, -c_0+c_2+t_2-t_3, 0, t_4, -c_2+c_4+t_3-t_4, \\
      & t_5, -c_2+c_6+t_3-t_5,...)^{T},
\end{split}
\]
where $t_0,t_1,t_2,...$ are free variables in $\mathcal{F}$.


\begin{thebibliography}{1}
\bibitem{Ha:Ch} Hart, R., The Chinese Roots of Linear Algebra, Johns Hopkins University Press (2010).
\bibitem{Gr:GE} Grcar, Joseph F., How ordinary elimination became Gaussian elimination, Historia Math. 38 (2011), no. 2, 163–218.
\bibitem{To:Auf} Toeplitz, O.,  \"{U}ber die Aufl\"{o}sung Unendichveiler Linearer mit unendlicheveilen Unbekannten, Pal. Rend., Vol. 28 (1909), 88-96.
\bibitem{Fu:Th} Fulkerson, D. R., \emph{Quasi-Hermite Forms of Row-finite Matrices}, Ph.D Thesis, Univercity of Wisconsin (1951), (available at http://www.researchgate.net/publication/36218809)
\bibitem{Os:Pr} Osofsky, Barbara L., Projective dimension is a lattice invariant, ePrint arXiv:math/0007091 (2000) . 
\bibitem{Pa:reap} Paraskevopoulos, A. G., A recursive approach to the solution of abstract linear equations and the Tau Method, Computers Maths. Applic., vol. \textbf{47}, No 10-11 (2004), 1753-1774.
\bibitem{Or:tau} Ortiz, E. L., The Tau Method, SIAM J. Numer. Anal. \textbf{6} (1969), 480-491.
\bibitem{AF:rings} Anderson, F. W., Fuller, K. R., \emph{Rings and Categories of Modules}, sec. ed., Springer (1992).
\bibitem{Bou:Alg} Bourbaki, N., \emph{Elements of Mathematics, Algebra I}, Hermann, Paris (1974).
\end{thebibliography}
\end{document}